\documentclass{amsart}
\usepackage{amsmath}
\usepackage{amsthm}
\usepackage{amsfonts}
\usepackage{amssymb}

\renewcommand\P{{\mathbf{P}}}

\newcommand\GAP{{\operatorname{Q}}}
\newcommand\R{{\mathbf{R}}}
\newcommand\Z{{\mathbf{Z}}}
\newcommand\C{{\mathbf{C}}}
\newcommand\E{{\mathbf{E}}}
\newcommand\D{{\mathbf{D}}}
\newcommand\I{{\mathbf{I}}}
\renewcommand\v{{\mathbf{v}}}
\renewcommand\u{{\mathbf{u}}}
\newcommand\V{{\mathbf{V}}}
\newcommand\w{{\mathbf{w}}}
\newcommand\N{{\mathcal{N}}}

\newcommand\tr{{\operatorname{tr}}}

\renewcommand\th{{\operatorname{th}}}

\newcommand\mes{{\operatorname{mes}}}
\renewcommand\Re{{\operatorname{Re}}}
\renewcommand\Im{{\operatorname{Im}}}

\renewcommand\a{{x}}
\renewcommand\b{{y}}
\newcommand\eps{\varepsilon}

\theoremstyle{plain}
  \newtheorem{theorem}[subsection]{Theorem}

  \newtheorem{lemma}[subsection]{Lemma}
  \newtheorem{corollary}[subsection]{Corollary}
   
\theoremstyle{remark}
  \newtheorem{remark}[subsection]{Remark}
  
   \newtheorem{example}[subsection]{Example}
\theoremstyle{definition}
  \newtheorem{definition}[subsection]{Definition}
\include{psfig}
\begin{document}

\title[The circular Law]{Random Matrices: The circular Law}

\author{Terence Tao}
\address{Department of Mathematics, UCLA, Los Angeles CA 90095-1555}
\email{tao@math.ucla.edu}
\author{  Van  Vu}
\address{Department of Mathematics, Rutgers, Piscataway, NJ 08854}
\email{vanvu@math.rutgers.edu}


\begin{abstract}  Let $\a$ be a  complex random variable with mean zero and bounded variance $\sigma^{2}$.
Let $N_{n}$ be a random matrix of order $n$ with  entries being i.i.d. copies of $\a$. Let $\lambda_{1}, \dots, \lambda_{n}$ be the eigenvalues of
$\frac{1}{\sigma \sqrt n}N_{n}$. Define the \emph{empirical spectral distribution} $\mu_{n}$ of  $N_{n}$ by the formula
$$ \mu_n(s,t) := \frac{1}{n} \# \{ k \leq n| \Re(\lambda_k) \leq s; \Im(\lambda_k) \leq t \}.$$

The following well-known conjecture has been open since the 1950's:

\vskip2mm

{\it Circular law conjecture:} $\mu_{n}$   converges
to the uniform distribution $\mu_\infty$ over the unit disk as $n$ tends to infinity.

\vskip2mm

We prove this conjecture, with strong convergence, under the
slightly stronger assumption that the $(2+\eta)\th$-moment of $\a$
is bounded, for any $\eta >0$. Our method builds and improves upon
earlier work of Girko, Bai, G\"otze-Tikhomirov, and Pan-Zhou, and
also applies for sparse random matrices.

The new key ingredient in the paper is a general result about the
least singular value of random matrices,
which was obtained using tools and ideas from additive combinatorics.

\end{abstract}

\maketitle

\section{Introduction}\label{sec1}

 Let $\a$ be a  complex random variable with finite non-zero variance $0 < \sigma^{2} < \infty$ and
  $N_{n}$ be the  random matrix of order $n$ with  entries being i.i.d. copies of $\a$.
  Let $\lambda_{1}, \dots, \lambda_{n}$ be the eigenvalues of
$\frac{1}{\sigma  \sqrt n} N_{n}$.
Define the \emph{empirical spectral distribution} (ESD) $\mu_{n}$ of  $N_{n}$ by the formula
$$ \mu_n(s,t) := \frac{1}{n} \# \{ k \leq n| \Re(\lambda_k) \leq s; \Im(\lambda_k) \leq t \}.$$
We say that the (strong) \emph{circular law} holds for $\a$ if, with probability $1$,
the spectral distribution $\mu_{n}$ converges
(uniformly) to the uniform distribution
$$\mu_\infty(s,t) := \frac{1}{\pi} \mes( \{ z \in \C| |z| \leq 1; \Re(z) \leq s; \Im(z) \leq t \} )$$
over the unit disk as $n$ tends to infinity.  In the literature one also
sees the \emph{weak circular law}, which asserts that for any fixed $s$
and $t$, that $\mu_n(s,t)$ converges to $\mu_\infty(s,t)$ in probability.

As the name suggests, the weak circular law is easier to prove than
the strong one. Using the approach in \cite{bai}, the proofs of both
types of convergence boil down to controlling the least singular value of
$\frac{1}{\sigma  \sqrt n} N_{n} - z I$. For the weak convergence,
one  needs a bound with failure probability tending to zero with $n$
tends to infinity (this is the approach taken in \cite{gotze, gt2},
for example). On the other hand, for the strong convergence one
needs the failure probability be summable in $n$. This appears much
more difficult and we  will discuss it in more detail in Section 2
(see the paragraph following Theorem \ref{lsv0}).


In this paper we shall be concerned exclusively with the strong circular law, and in particular with regard to the following well-known conjecture:

\vskip2mm

{\it Circular law conjecture.}  The strong circular law holds for any complex variable $\a$ with zero mean and finite non-zero variance.
\vskip2mm


The circular law conjecture was formulated in the early 1950s, as a
natural (non-hermitian) counterpart of Wigner's semi-circle law.
Since then, several partial results have been obtained, at the cost
of extra assumptions on the distribution of the basic variable $\a$.
In the next few paragraphs, we give a brief survey of these results.

If  $\a$ is complex Gaussian, the conjecture was proved by Mehta \cite{meh} in 1967, using the joint density function of the eigenvalues $\lambda_{i}$ which
was discovered by Ginibre few years earlier \cite{gin}.  An important breakthrough was made by  Bai \cite{bai-first}, following an earlier work of Girko \cite{girko}.
(Bai's paper discussed Girko's paper carefully and pointed out some gaps in that paper.)
 In \cite{bai-first}, Bai proved the claim under the assumption that $\a$ has finite sixth moment ($\E |\a|^6 < \infty$) and that the joint distribution of the real and imaginary parts of $\a$ has a bounded density.  Recently,  in \cite[Chapter 10]{bai},  a finer result
was obtained showing  that the sixth moment hypothesis can be weakened to $\E |\a|^{2+\eta} < \infty$ for any specified $\eta > 0$. However, the bounded density assumption remains critical. This assumption, unfortunately, excludes several important distributions, for instance discrete distributions such as Bernoulli random variables $x \in \{-1,+1\}$.

\begin{theorem} \label{theorem:bai}  \cite[Theorem 10.3]{bai} Assume that the complex random variable $\a$ has zero mean and  finite $(2+\eta)^\th$ moment for
some $\eta >0$ and also that the joint distribution of the real and imaginary part has a bounded density. Then the circular law holds for $\a$.
\end{theorem}

A key idea in \cite{bai-first} is to analyze the ESD $\mu_n$ through its \emph{Stieltjes transformation} $s_n: \C \to \C$, defined by the formula\footnote{We are using $\sqrt{-1}$ for the imaginary unit, as we wish to reserve $i$ as an index of summation.}
$$s_{n}(z):= \frac{1}{n} \sum_{k=1}^{n}  \frac{1}{\lambda_{k} -z } = \int_\C \frac{1}{s+\sqrt{-1} t-z}\ d\mu_n(s,t).$$

As $s_{n}(z)$ is analytic everywhere except the poles, the real part already determines the eigenvalues $\lambda_{k}$.
If write $s_{n} (z) = s_{nr} (z) + \sqrt{-1} s_{ni} (z)$, $\lambda_{k} =\lambda_{kr} + \sqrt{-1} \lambda_{ki} $ and $z= s + \sqrt{-1}t$, we have the important identity

$$ s_{nr} (z) = \frac{1}{n} \sum_{k=1} ^{n } \frac{\lambda_{kr} -s} {|\lambda_{k} -z|^{2}} = -\frac{1}{2n} \sum_{k=1}^{n} \frac {\partial}{\partial s} \log |\lambda_{k}-z|^{2} = -\frac{1}{2} \frac {\partial}{\partial s} \int_{0}^{\infty} \log x  \nu_{n} (dx,z), $$

\noindent where $\nu_{n} (.,z)$ is the ESD of the Hermitian matrix $H_{n }:= (\frac{1}{\sqrt n} N_{n}- zI) (\frac{1}{\sqrt n} N_{n}- zI) ^{\ast}$. The task then reduces (at least in principle) to controlling the distributions $\nu_n$.

The $\log$ function has two poles, at $\infty$ and $0$. The first one is easy to deal with, as one can bound the largest singular value by a polynomial in $n$.
  The pole at $0$ poses a much more serious obstacle,  since  the smallest  eigenvalue of $H_{n}$ (or the least singular value of
$N_{n}-zI$) can be arbitrary close to $0$. (In fact, if the matrix is singular, which happens with positive probability in discrete models, then the least singular value is $0$.)
The bounded density  assumption  in Theorem \ref{theorem:bai} was introduced primarily in order to handle this obstacle.

In the last few years,  the least singular value problem has become
better understood in the discrete case, thanks to a series of papers
\cite{TVsing, Rud, RV, TVstoc}. In these papers, strong lower bounds
for the least singular value of a random matrix \cite{TVsing, Rud,
RV} or a random perturbation of a fixed matrix \cite{TVstoc} were
obtained. As a consequence, the circular law has recently been
established for various new classes of distributions.  For instance,
G\"otze and Tikhomirov \cite{gotze} proved  the weak circular law
for any sub-Gaussian\footnote{A variable is sub-Gaussian if it has
exponential tail; in particular all of its moments are bounded.}
distribution $\a$, using the arguments from  \cite{Rud}.   In
\cite{twenty}, Girko established the weak circular law assuming
bounded $4+\delta$ moment for some $\delta > 0$. Relying on
\cite{RV}, Pan and Zhou \cite{PZ} were recently able to verify the
strong  circular law for any distribution with a bounded fourth
moment. This assumption is needed for a number of reasons, in
particular allowing one to bound the operator norm of $N_n$ by
$O(\sqrt{n})$ with high probability. Very recently (a few months after
the current paper was first posted on the arXiv), G\"otze and Tikhomirov
\cite{gt2} proved the weak circular law under an assumption similar
to our main theorem below.

In this paper, we prove  the circular law only assuming a  bounded $(2+\eta)^\th$  moment,
for any fixed $\eta >0$.  In particular, we
have completely removed the
bounded density function assumption in Theorem \ref{theorem:bai}.

\begin{theorem}[Circular law]\label{circular}  Assume that $\a$ is a complex random variable
with zero mean and finite $(2+\eta)^\th$ moment for
some $\eta >0$, with strictly positive variance. Then the strong circular law holds for $\a$.
\end{theorem}


This result can be further strengthened in several directions:

\begin{itemize}

\item We can further
 relax the condition $\E |\a|^{2+\eta} < \infty$ to $\E |\a|^2 \log^C(2+|\a|) < \infty$, where $C$ is a sufficiently large absolute constant.
  (For instance, $C=16$ is sufficient; see Section \ref{relax-sec} for details.)

  \item It is not necessary to assume that the entries have identical distributions. It suffices to assume that they are independent, have mean zero with uniformly bounded $(2+\eta)^\th$ moments, and that they are all \emph{dominated} (in a Fourier-analytic sense) by a single random variable with finite non-zero variance and bounded $(2+\eta)^{\th}$ moment; see Remark \ref{dom} below.  (See also \cite[p. 326-327]{bai} in which the extension of the circular law to the case of non-identical distributions is discussed.)

  \item One can obtain some quantitative estimates on the rate of convergence as well. For example, under the $(2+\eta)^\th$ moment assumption, we can show that almost surely,
  the distance $\sup_{s,t} |\mu_{n}(s,t)-\mu_\infty(s,t)|$ between $\mu_{n}$ and the limiting distribution $\mu_\infty$ in the uniform metric is at most $n^{-\eta'}$ for some constant $\eta' > 0$ and all sufficiently large $n$.

  \end{itemize}

  It is too technical to address these points  in the main proof, so we are going to first prove Theorem \ref{circular} and sketch out the necessary modifications to obtain these refinements in Sections \ref{relax-sec}, \ref{sec14}.

   The circular law  also holds for \emph{sparse} random matrices. For $0<\mu \le 1$, let $\I_{\mu}$ be the boolean
  random variable which takes value 1 with probability $\mu$ and $0$ with probability $1-\mu$.
  Let $\rho= n^{-1+\alpha}$, for a positive constant $\alpha$.   Let $N_{n,\rho}$ be the random matrix with the  $ij$ entry being $\I_{i,j, \rho} \a_{i,j}$, where the $\I_{i,j,\rho}$ and $\a_{ij}$ are jointly independent iid copies of $\I_{\rho}$ and $\a$ respectively.
   G\"otze and Tikhomirov \cite{gotze} proved that if $\a$ is sub-Gaussian and $\alpha > 3/4$, then $N_{n, \rho}$ admits the circular law. We can prove
   the following strengthening of this result:

 \begin{theorem}[Circular law for sparse matrices] \label{circular-sparse}
 Let $\alpha > 0$ and $\eta > 0$ be arbitrary positive constants. Assume that $\a$ is a complex random variable with zero mean and finite $(2+\eta)^\th$ moment.
 Set $\rho= n^{-1+\alpha}$ and let  $\mu_{n, \rho}$ be the ESD of $\frac{1}{\sigma \sqrt {n \rho} }N_{n,\rho}$, where $\sigma^2$, as usual, is the variance of $\a$. Then $\mu_{n, \rho}$ tends to the uniform distribution $\mu_\infty$ over the unit disk as $n$ tends to infinity.
\end{theorem}

  \begin{remark} If one takes $\alpha=0$, the circular law no longer holds. In this case, $\rho=n^{-1}$ and
  each entry equals $0$ with probability $1-1/n$. Thus, a row is all-zero with probability
  $(1-1/n)^{n} \approx e^{-1}$. Since the rows are independent, it is easy to show that with high probability one has $\Theta (n)$ all-zero rows. But this means that  the ESD,
  with high probability, has positive constant mass at the origin. \end{remark}

We shall prove this theorem in parallel with Theorem \ref{circular}, by indicating at various junctures what the ``sparse'' version of certain key lemmas are.

The key ingredient in our proof of the circular law
is a new lower bound for the least singular value of the matrix $M+N_{n}$, where $M$ is an arbitrary matrix with complex entries having absolute values bounded from above by a polynomial in $n$.  For the circular law, we only need to consider the case $M=-zI$, where $I$ is the identity matrix. On the other hand, the general case is interesting on its own right and proves useful in other areas of mathematics (see, for example, \cite{TVstoc}).  Our arguments permit the coefficients of $M$ or $N_n$ to be as large as $n^C$ for any fixed constant $C$, which is the main reason why we do not need any stronger moment control on $\a$ beyond the $(2+\eta)^\th$ moment.

The rest of the paper is organized as follows. In the next section, we present the above mentioned result on the least singular value. The key tool for proving this result is a so-called \emph{Inverse Littlewood-Offord theorem}, discussed in Section \ref{sec3}. This theorem is motivated by  several previous results of the same spirit
from  \cite{TVsing}.  On the other hand, the bound in Section \ref{sec3} is nearly optimal and is  sharper than one that can be deduced from \cite{TVsing}.  This improvement is critical to us.

The proof of the Inverse Littlewood-Offord theorem is technical and
requires several lemmas, developed in Sections
\ref{concprob}-\ref{sec9}. In particular, we prove a forward
Littlewood-Offord theorem (Theorem \ref{flot}), which seems to be of
interest on its own right.  The proof of the Inverse Theorem follows
in Section \ref{sec10}. Next, we prove the desired bound on the
least singular value in Section \ref{lsv-sec}. The proof of the
circular law follows in Section \ref{circular-sec}. The rest of the
paper is devoted to various refinements of the circular law; for
instance, Theorem \ref{circular-sparse} is discussed in Section
\ref{sparse-sec}.

In order to handle the sparse case, we need  sparse versions of all
the tools mentioned above. These results can be proved using the
same argument with some modifications. We will only sketch these
proofs in the paper.

Let us conclude this section with our notation.

\begin{definition}[Asymptotic notation]  In the whole paper we assume that $n$ is sufficiently large,
whenever needed. Asymptotic notation is used under the assumption
that $n \rightarrow \infty$. Let $X$ and $Y$ be non-negative
quantities. $X = O(Y)$, $X \ll Y$, $Y \gg X$ and $Y = \Omega(X)$ all
mean that
 $X \leq CY$ for some positive constant $C$ and
  $X = \Theta(Y)$  means  $X \ll Y \ll X$; $X = o(Y)$ means that  $|X| \leq c(n) Y$ where $c(n)$  goes to zero as $n \to \infty$.

  In many cases, we want to indicate that  the hidden constants in $O, \Omega, \Theta$ or $\ll, \gg$
    depend on some additional parameters. In such cases,
     we will indicate this by subscripts. For instance,  $X = O_\eps(Y)$
     means that there is a positive constant  $C(\eps)$  depending only on $\eps$ such that $ X\le C(\eps) Y$.

  Throughout the paper, letters $A,B, C,c, \alpha, \eps, \eta, \delta, \kappa$ are used to denote constants. Letters $\mu, \rho, \beta$ denote quantities that may depend on $n$.
  \end{definition}

We use $\P$ to denote probability, $\E$ to denote expectation, and $\I_\rho$ to
denote indicator functions of expectation $\rho$ as used earlier in this section.
 If $E$ is an event, we use $\I(E)$ to denote the indicator of $E$, which equals $1$ when
 $E$ is true and $0$ otherwise.  The cardinality of a finite set $S$ will be denoted $\# S$,
 and the Lebesgue measure of a set $A \subset \C$ will be denoted $\mes(A)$.

\section{Least singular value bound}\label{sec2}

Let $M$ be a matrix of order $n$.
We use $\|M\|$ to denote the spectral norm of $M$ (i.e. the largest singular value of $M$)
$$\| M\| = \sup_{|\v|=1} |M\v|. $$

As discussed in the previous section, a key point in Bai's approach
is to obtain control on the lower tail distribution for the least singular value of $\frac{1}{\sqrt{n}} N_n - z I_n$, or equivalently to obtain control on the upper tail distribution of the norm of the inverse $\|(\frac{1}{\sqrt{n}} N_n - z I_n)^{-1}\|$.

This will be achieved in Theorems \ref{lsv0} below. The strength of this theorem  is that it
requires a very weak assumption on the distribution of the entries. All we need is a  finite second moment.
Several results of this type were obtained recently, under stronger assumptions of $\a$.
For example,  \cite{SST} addressed the case when $\a$ was real Gaussian; \cite{TVstoc}
addressed the case when $\a$ has support on the integers and $M$ has integer entries.  This was done building upon the $M=0$ case discussed in \cite{TVsing}.
 The case when   $\a$ has finite third moment  and that  $\| M \|$
 is bounded by $O(n^{1/2})$  was addressed  in \cite{PZ} (building upon the $M=0$ real-valued case proven in \cite{RV}).  In this result, the assumption on the norm of $M$ is important and the constant $1/2$ (in the exponent of $n$) cannot be replaced by any other constant.
 Furthermore, in the complex-valued case, the bounds in \cite{PZ} depended on the entire covariance matrix of $\a$ and not just on the variance.

\begin{theorem}[Least singular value bound]\label{lsv0}  Let $A, C_{1}$ be positive constants,
and let $\a$ be a complex-valued random variable with non-zero
finite variance (in particular, the second moment is finite). Then
there are positive constants $B$ and $C_{2}$  such that the
following holds: if $N_{n}$ is the random matrix of order $n$ whose
entries are iid  copies of $\a$, and $M$ is a deterministic matrix
of order $n$ with spectral norm at most $n^{C_{1}}$, then,
$$ \P( \| (M + N_n)^{-1} \| \geq n^B ) \le C_{2} n^{-A}.$$
\end{theorem}

It is very important that we can have any constant $A$ in the bound.
If   $A>1$, then the right hand side is summable in $n$ and this is
critical to the strong circular law. In order to prove the weak law,
any $A$ suffices. The difficulty between getting any $A$ and getting
$A>1$ can be illustrated by the following simplified case. Take $M$
be the zero matrix and $N$ be the random Bernoulli matrix (whose
entries take value $\pm 1$ with probability $1/2$). To make the
situation even simpler, assume that we only want to bound the
probability that $N^{-1}$ does not exists (namely that $N$ is
singular). Already in the  70s, Koml\'os \cite{Bolbook} proved that
this probability is $O(n^{-1/2})$. However, the first proof for a
bound of the type $O(n^{-1 -\eps})$ was obtained only almost twenty
years later by Kahn, Koml\'os and Szemer\'edi  \cite{KKS}, using a
much more complex argument.

Let us now go back to Theorem \ref{lsv0}. In fact, we have a more
precise statement involving a seemingly stronger (but actually
equivalent) assumption on $\a$. More precisely, we introduce the
following technical definition.

\begin{definition}[Controlled second moment]  Let $\kappa \geq 1$.
A complex random variable $\a$ is said to have \emph{$\kappa$-controlled second moment} if one
has the upper bound
$$ \E |\a|^2 \leq \kappa$$
(in particular, $|\E \a| \leq \kappa^{1/2}$), and the lower bound
\begin{equation}\label{eyot}
 \E \Re( z \a - w )^2 \I(|\a| \leq \kappa) \geq \frac{1}{\kappa} \Re(z)^2
\end{equation}
for all complex numbers $z,w$.
\end{definition}

\begin{example} The Bernoulli random variable ($\P(\a=+1) = \P(\a=-1) = 1/2$) has $1$-controlled second moment.  The condition \eqref{eyot} asserts in particular that $\a$ has variance at least $\frac{1}{\kappa}$, but also asserts that a significant portion of this variance occurs inside the event $|\a| \leq \kappa$, and also contains some more technical phase information about the covariance matrix of $\Re(\a)$ and $\Im(\a)$.
\end{example}

 To show that this condition is not
significantly stronger than bounded second moment, we  prove
 that any complex random variable of finite non-zero variance has controlled second moment after a (harmless) phase rotation:

\begin{lemma}\label{rot} Let $\a$ be a complex random variable with finite non-zero variance.
 Then there exists a phase $e^{\sqrt{-1} \theta}$ and a $\kappa \geq 1$ such that $e^{\sqrt{-1} \theta} \a$ has $\kappa$-controlled second moment.
\end{lemma}

\begin{proof}  For $\kappa$ sufficiently large, we have $\E |\a|^2 \leq \kappa$, and the
event $|\a| \leq \kappa$ has probability at least $1/\sqrt{\kappa}$.
Let $\a_\kappa$ be the variable $\a$ conditioned on the event $|\a|
\leq \kappa$.  Since $\a$ has non-zero variance, we see that
$\a_\kappa$ will also have non-zero variance for $\kappa$ large
enough.  It will then suffice to show that
$$  \E \Re( z \a_\kappa - w )^2 \geq \Re(z)^2 \frac{1}{\sqrt{\kappa}}$$
after rotating $\a$ by a phase if necessary.  If we write $\b_\kappa := \a_\kappa - \E(\a_\kappa)$,
then we easily compute
$$\E \Re( z \a_\kappa - w )^2 = \E \Re( z\b_\kappa + z\E(\a_\kappa)- w)^2 = \E \Re(z\b_\kappa)^2 + \Re(z\E(\a_\kappa)-w)^2$$
so it suffices to show that for $\kappa$ sufficiently large we have
\begin{equation}\label{zbk}
  \E \Re( z \b_\kappa )^2 \geq \Re(z)^2 \frac{1}{\sqrt{\kappa}}.
\end{equation}
Now set $\b := \a - \E(\a)$ and consider the covariance matrix
$$ \left( \begin{matrix} \E \Re(\b)^2 & \E \Re(\b) \Im(\b) \\ \E \Re(\b) \Im(\b) & \E \Im(\b)^2 \end{matrix} \right).$$
Since $\a$ has finite non-zero variance, we see that this matrix is finite, non-zero, and positive semi-definite. In particular its largest eigenvalue is at least $\delta$ for some $\delta > 0$.  By monotone convergence we then conclude that the covariance matrix
\begin{equation}\label{cov} \left( \begin{matrix} \E \Re(\b_\kappa)^2 & \E \Re(\b_\kappa) \Im(\b_\kappa) \\ \E \Re(\b_\kappa) \Im(\b_\kappa) & \E \Im(\b_\kappa)^2 \end{matrix} \right)
\end{equation}
has largest eigenvalue at least $\delta/2$ for $\kappa$ sufficiently large.

Now fix $\kappa$ large enough so that all the above statements hold,
and also so that $\frac{1}{\sqrt{\kappa}} \leq \frac{\delta}{2}$.
The null space of \eqref{cov} is at most one-dimensional.  By
rotating $\a$ by a phase we may then assume that the null space is
contained in the imaginary axis $\{ \left( \begin{matrix} 0 \\ w
\end{matrix} \right): w \in \R \}$.  Since covariance matrices are
positive semi-definite, we thus have the quadratic form estimate
$$ |u^2 \E \Re(\b_\kappa)^2 + 2uv \E \Re(\b_\kappa) \Im(\b_\kappa) +
 v^2 \E \Im(\b_\kappa)^2| \geq \frac{\delta}{2} u^2,$$
and \eqref{zbk} follows by setting $u= Re(z)$ and $v= Im (z)$.
\end{proof}

Since rotating all entries by the same phase does not change the norm of the inverse,
Theorem \ref{lsv0}  follows from the following theorem.

\begin{theorem}[Least singular value bound]\label{lsv}  Let $A, C_{1}, C_{2}$ be positive constants. There are positive constants $B$ and $C_{3}$  such that the following holds.
Let $\a$ be a random variable with $C_{1}$-controlled second moment and $N_{n}$ be the random matrix of order $n$ whose entries are i.i.d copies of
$\a$. Let $M$ be a deterministic matrix of order $n$ with spectral norm at most $n^{C_{2}}$. Then,
$$ \P( \| (M + N)^{-1} \| \geq n^B ) \le  C_{3} n^{-A}.$$
\end{theorem}
\begin{remark}

   Our arguments give an explicit dependence of $B$ in terms
    of $A$ and $C_{2}$. One can set $B$ to be roughly $2AC_{2} $. A more exact
    dependence can be obtained with considerably more technical
   details. Since for the proof of the circular law, any constant $B$ suffices,
   we do not go into this matter here and will discuss it elsewhere.
\end{remark}

\begin{remark}
Notice that the assumptions in Theorem \ref{lsv} are weaker than the assumption of Theorem \ref{circular}.  We do not require $\a$ to have mean $0$ and bounded
$(2+\eta)^\th$ moment.
In the proof of  Theorem \ref{circular}, these extra assumptions are  needed in order to
repeat the approach of Bai, and are unrelated to the pole problem or Theorem \ref{lsv}.
\end{remark}

\begin{remark}\label{dom} One can relax somewhat the hypothesis that the entries $\a_{ij} $ of $N_{n}$ are i.i.d copies of  $\a$. It is sufficient to assume the following

\begin{itemize}
\item   $\a_{ij}$  are  \emph{dominated} by a single distribution $\a$ in the Fourier-analytic sense that $|\E( e^{2\pi \sqrt{-1} \Re( \xi \a_{ij} )} )| \leq \E( e^{2\pi \sqrt{-1} \Re( \xi \a )} )$ for all complex numbers $\xi$.
\item   $\a$ has  $\kappa$-controlled second moment for some fixed $\kappa$.
\end{itemize}

 This refinement can be extracted without too much difficulty from the proof in this paper, which ultimately relies on Fourier-analytic methods.
 Using this refinement and following \cite[Chapter 10.8.2]{bai},  we can extend Theorem \ref{circular} for the case the the entries of $N_{n}$ are independent, but not necessarily identically distributed, as mentioned in the introduction.
  \end{remark}

  In order to deal with sparse random matrices, we prove the following variant of Theorem \ref{lsv0}.

  \begin{theorem} [Least singular value for sparse matrices] \label{lsv-sparse}
   Let $A>1, C_{1}, C_{2}, \alpha$ be positive constants. There are positive constants $B$ and $C_{3}$ depending on $A,C_{1}, C_{2}, \alpha$ such that the following holds.
Let $\a$ be a random variable with $C_{1}$-controlled second moment
and let $N_{n, \rho}$ be the random matrix of order $n$ defined as
in Theorem \ref{circular-sparse}. Let $M$ be a deterministic matrix
of order $n$ with spectral norm at most $n^{C_{2}}$. Then,
$$ \P( \| (M + N_{n,\rho})^{-1} \| \geq n^B ) \le C_{3} n^{-A}.$$
  \end{theorem}


To conclude this section, let us derive a simple corollary of Theorem \ref{lsv-sparse}.

\begin{corollary}[Condition number bound]   Let $A, C_{1}, C_{2}, \alpha$ be positive constants. There are positive constants $B$ and $C_{3}$  such that the following holds.
Let $\a$ be a random variable with $C_{1}$-controlled second moment and $N_{n, \rho}$ be the random matrix of order $n$
defined as in Theorem \ref{circular-sparse}. Let $M$ be a deterministic matrix of order $n$ with spectral norm at most $n^{C_{2}}$. Then,

$$ \P( \| M+N_{n,\rho} \| \| (M + N_{n,\rho})^{-1} \| \geq n^B ) \le C_{3} n^{-A}.$$
\end{corollary}

\begin{proof} A simple application of Chebyshev's inequality shows that $$\P( |\a| \geq n^{A/2 + 1} ) \ll_{C_1} n^{-A-2}. $$
Since $\|N_{n,\rho}\|$ is bounded from above by $\max_{i} \sum_{j=1}^{n } |\a_{ij}|$,  we have that
 $$\P( \|N_{n,\rho}\| \geq n^{A/2 + 2} ) \ll_{C_1} n^{-A}$$  by the union bound.
 Combining this with the polynomial bound on $\|M \|$ and with Theorem \ref{lsv-sparse}, the claim follows by choosing $B$ sufficiently large.
\end{proof}

The condition number $\|M \| \|M^{-1}\|$ of a matrix $M$ plays a crucial role in numerical linear algebra (see \cite{BT}, for instance). The above corollary implies that if one perturbes a fixed matrix $M$ by a (very general) sparse random matrix $N_{n}$,  the condition number of the resulting matrix will be relatively small with high probability. This fact has some nice applications in theoretical computer science (see \cite{TVstoc} or \cite{ST}, for example).

\section{Inverse Littlewood-Offord theorems}\label{sec3}

Let us consider a toy case in order to illustrate the ideas behind the proof of Theorem \ref{lsv}. Assume, for a moment, that $M=0$ and $\a \equiv N(0,1)$ is real Gaussian.
In this case, we talk about the least singular value of the random matrix $N_{n}$ whose entries are i.i.d real Gaussian.
Let $X_{i}$ be the row vectors of $N_{n}$ and $d_{i}$ be the distance from $X_{i}$ to the hyperplane $H_{i}$  spanned by $X_{j}$, $j \neq i$. The least singular value of
$N_{n}$ is close (up to factors of $n^{O(1)}$) to $\min_{1\le i\le n} d_{i}$. Thus, our goal is to prove that with high probability, each of the $d_{i}$ is bounded away from 0.

In this Gaussian case, the task is simple since, thanks to symmetry, the distribution of $d_{i}$ does not depend on the vectors $X_{j}$, $j \neq i$. Indeed, $d_i$ has the same distribution as the distance from a Gaussian vector to a fixed hyperplane. This variable is well understood and satisfies the inequality

$$\P (d_{i} \le  n^{-A-1/2} ) = O(n^{-A}) $$

\noindent for any fixed positive constant $A$.  This leads to the conclusion of  Theorem \ref{lsv} in this simple case.

However, the general case is much more difficult. For example, if the entries of $N$ are iid Bernoulli, it is already non-trivial to prove $N_{n}$ is asymptotically almost surely non-singular (i.e. that with probability $1-o(1)$, one has $d_i \neq 0$ for all $i$).
The point here is that one can no longer fix $X_{j}, j \neq i$. As a matter of fact, the distribution of the distance $d_{i}$ depends heavily on the position of the hyperplane $H_{i} $ spanned by the $X_{j}, j \neq i$.  For example, let $\a$
be Bernoulli and consider the following two situations

\begin{itemize}

\item $H_{i}$ has normal vector $(\frac{1}{\sqrt n},  \cdots, \frac{1}{\sqrt n})$. In this case, $d_{i}=0$ with probability $O(\frac{1}{\sqrt n})$.

\item $H_{i}$ has normal vector $(\frac{1}{\sqrt 2}, \frac{1}{\sqrt 2}, 0, \dots, 0)$.  In this case, $d_{i}=0$ with probability $\frac{1}{2}$.
\end{itemize}

A hyperplane $H$ is, in some sense, {\it bad } for us  if  the distance from a random (row) vector to $H$ is small with non-negligible probability.  It is important to understand the bad hyperplanes. Notice that if $\v=(v_{1}, \dots, v_{n})$ is the unit  normal vector of $H$, then the distance  in question is exactly the random variable

$$  | v_1 \a_1 + \ldots + v_n \a_n|, $$

\noindent where $\a_{i}$ are i.i.d. copies of $\a$.

This naturally leads to introducing the following concept.

\begin{definition}[Small ball probability]\label{sbp}  Let $\a$ be a complex random variable, and let
$\v = (v_1, \ldots, v_n)$ be a tuple of complex numbers.  We define the \emph{random walk} $W_{\a}(\v)$ to be the complex random variable
\begin{equation}\label{wav}
 W_\a(\v) := v_1 \a_1 + \ldots + v_n \a_n
\end{equation}
where $\a_1,\ldots,\a_n \equiv \a$ are iid copies of $\a$.  For any $z \in \C$ and $r > 0$, we let
 $B(z,r)$ denote the closed disk of radius $r$ centered at $z$.
  For any $r \geq 0$, we define the \emph{small ball probability}
$$ p_{r,\a}(\v) := \sup_{z \in \C} \P( W_\a(\v) \in B(z,r) ). $$
\end{definition}
Intuitively, we expect the small ball probability $p_{r,\a}(\v)$ to be quite small for ``most'' tuples $\v$.  The question, of course, is to quantify ``most''.

A classical theorem of Littlewood and Offord  \cite{LO} (see also \cite{erdos-lo}) shows that if $\a$ is Bernoulli, and all $|v_{i}| \ge 1$, then
$p_{1, \a} (\v) = O(n^{-1/2})$. There are several extensions of this result. They, typically, improve upon the bound $O(-n^{1/2})$, under extra assumptions on the  $v_{i}$. We are going to refer to results in this spirit as {\it forward} Littlewood-Offord theorems.

For our purposes, we  need {\it inverse} Littlewood-Offord theorems. Such a theorem
is supposed to give a  characterization of those vectors $\v$, where  $p_{r,\v}$ is larger than some lower bound.
The study of inverse Littlewood-Offord theorems was started in \cite{TVsing}, where
 we investigated the case when
$\a$ has discrete support. A new result in this spirit was recently obtained in \cite{RV}, where the authors investigated sub-Gaussian distributions, as well as distributions with bounded third or fourth moments.

In the current situation, we only assume that $\a$ has $O(1)$-controlled second moment.  The weakness of this assumption is a major obstacle and makes the proof much more
complicated. It is still possible to obtain a reasonably strong characterization of $\v$, given that $p_{r, \a} (\v)$ is large. However, this characterization
is somewhat technical to state and so we will only explicitly state here a corollary of it, which will be sufficient for our purpose of proving the least singular value bound and the circular law.

Let $\a$ be a complex random variable. Let $n$ be a positive integer and $\beta, p$ be positive numbers that may depend on $n$.
 Let $S_{n, \a, \beta,p}$ be the set of all unit vectors $\v = (v_1, \ldots, v_n) \in \C^n$ such that one has the concentration bound
$$ p_{\beta,\a}(\v) \geq p.$$
We give $\C^n$ the $l^\infty$ norm
$$ \| (v_1,\ldots,v_n) \|_\infty := \sup_{1 \leq i \leq n} |v_i|.$$

\begin{theorem}[Inverse Littlewood-Offord theorem]\label{ilo}
Let $\a$ be a complex random variable which has $\kappa$-controlled
second moment for some $\kappa >0$. Let $0 < \eps \le 1$.  Then, for
all $n$ which are sufficiently large depending on $\kappa, \eps$ and
$\beta \ge \exp(- n^{\eps/2})$ and $p= n^{-O(1)}$, there is a set
$S' \subset \C^n$ of size at most $n^{(-1/2 +\eps)n } p^{-n} +
\exp(o(n))$ such that for any $\v \in S_{n, \a, \beta, p}$ there is
$\v' \in S'$ such that $\|\v-\v'\|_{\infty} \le \beta$. In other
words, $S_{n, \a, \beta, p}$ has a maximal $\beta$-net in the $l^\infty$
norm of size at most $ n^{(-1/2 +\eps)n } p^{-n} + \exp(o(n))$.
\end{theorem}

\begin{theorem}[Inverse Littlewood-Offord theorem for sparse random variables]\label{ilo-sparse}
Let $\a$ be a complex random variable which has $\kappa$-controlled
second moment for some $\kappa >0$. Let $0 < \eps \le 1$.  Then, for
all $n$ which are sufficiently large depending on $\kappa, \eps$ and
$\beta \ge \exp(-n^{\eps/2})$ and $p= n^{-O(1)}$, all $1/n < \mu \le
1$, and all $m$ between $n^\eps$ and $n^{1-\eps} \mu$ there is a set
$S' \subset \C^n$ of size at most $ n^{O(\eps)n} (p \sqrt{m})^{-n} +
\exp(n^{O(\eps)} m / \mu)$ such that for any $\v \in S_{n,
\a\I_{\mu}, \beta, p}$ there is $\v' \in S'$ such that such that
$\|\v-\v'\|_{\infty} \le \beta$. In other words, $S_{n, \a \I_{\mu},
\beta, p}$ has a maximal $\beta$-net in the $l^\infty$ norm of size at most
$ n^{O(\eps)n} (p \sqrt{m})^{-n} + \exp(n^{O(\eps)} m / \mu)$.
\end{theorem}

\begin{remark}  If one sets $m = n^{1-C\eps} \mu$ for some absolute constant $C$ then the conclusion of Theorem \ref{ilo-sparse} is similar to that in Theorem \ref{ilo} except for the extra term $\sqrt \mu$ in Theorem \ref{ilo-sparse}.  However, for our applications it will be slightly more convenient to choose $m$ at the other extreme, thus $m = n^\eps$.  The main point here is that the size of $S_{n,\a,\beta,p}$ (or $S_{n,\a \I_\mu,\beta,p}$) tends to be much smaller than $p^{-n}$.
\end{remark}

\begin{definition}[Entropy]  Let $A$ be a precompact subset of a metric space $X$, and let $\eps > 0$.  We define the \emph{internal metric entropy} $\N_\eps(A)$ to be the cardinality of the largest $\eps$-net in $A$ (i.e. a set $B \subset A$ where any two elements in $B$ are separated by distance $\eps$).  We define the \emph{external metric entropy} $\N'_\eps(A)$ to be the least number of closed $\eps$-balls in $X$ needed to cover $A$.
\end{definition}

One easily verifies that
$$ \N_{2\eps}(A) \leq \N'_\eps(A) \leq \N_\eps(A),$$
and furthermore in the complex plane $X=\C$ we have $\N_{2\eps}(A) =\Theta (\N_\eps(A))$. As constant factors will not play any important role,  the two notions of entropy will be essentially equivalent for our purposes.

Since $\| \v\| _{\infty} \ge n^{-1/2} \|\v \|$,  we have the
following corollary.

\begin{corollary}  \label{cor:lsv}
Let $\a$ be a complex random variable which has $\kappa$-controlled
second moment, for some constant $\kappa >0$. Let $\eps$ be an
arbitrary positive constant. Then for any positive numbers $\mu,
\beta, p \le 1$ and all sufficiently large $n$ we have
$$ \N_{\beta n^{1/2} }(S_{n, \a, \beta, p}) \le n^{(-1/2 +\eps)n } p^{-n} + \exp(o(n)) $$
\noindent and
$$ \N_{\beta n^{1/2} }(S_{n, \a \I_{\mu}, \beta, p} )\le n^{(-1/2 +\eps)n } (p \sqrt \mu)^{-n} + \exp(o(n)) $$
\end{corollary}

\begin{remark} In fact, the proof of Theorem \ref{ilo} gives a fairly precise description of the set $S_{n,\a,\beta,p}$, as is the case with other inverse Littlewood-Offord theorems in the literature.  However, this description is somewhat  technical to state and we only need the entropy bound on $S_{n,\a,\beta,p}$ in our application, so we have presented Theorem \ref{ilo} in the above short (but less explicit) form.
\end{remark}

\section{Concentration probabilities  and Fourier analysis}\label{concprob}  

Throughout this section $\a$ will be a fixed complex random variable with $O(1)$-controlled second moment.  For any $0 < \mu \leq 1$, let $\a^{(\mu)}$ be the random variable
\begin{equation}\label{amu}
 \a^{(\mu)} := (\a_1 - \a_2) \I_{\frac{\mu}{2} }
\end{equation}
where $\a_1, \a_2$ are iid copies of $\a$ and $\I_{\frac{\mu}{2}}$ is independent from $\a_{1}, \a_{2}$.

\begin{example} If $\a$ is the Bernoulli random variable $\P(\a = +1) = \P(\a = -1) = 1/2$, then $\a^{(\mu)} \in \{0,+2,-2\}$ with $\P(\a^{(\mu)} = +2) = \P(\a^{(\mu)} = -2) = \mu/8$.
\end{example}

For any $0 < \mu \leq 1$ and any tuple $\v = (v_1,\ldots,v_n)$ of complex numbers, define the \emph{concentration probability}
\begin{equation}\label{pumar}
\P_{\mu}(\v) := \E \exp( - \pi |W_{\a^{(\mu)}}(\v)|^2 ).
\end{equation}
This quantity will turn out to be very convenient for controlling the small ball probabilities of $W_{\a^{(\mu)}}(\v)$ (see Lemma \ref{concball} below).  To do that, we first need a Fourier-analytic representation of $\P_\mu(\v)$.  We introduce the \emph{characteristic function} $f: \C \to \R$, defined by
\begin{equation}\label{fw}
 f(z) := |\E( e( \Re( \a z ) ) )|^2
\end{equation}
where $e$ is the standard character
$$e(t) := e^{2\pi \sqrt{-1} t}.$$

\begin{lemma}[Fourier representation]\label{pf2}  For any tuple $\v = (v_1, \ldots, v_n)$ of complex numbers and any $0 < \mu \leq 1$, we have
\begin{equation}\label{pumar-fourier-2}
 \P_{\mu}(\v) = \int_\C \prod_{i=1}^n \left(1-\frac{\mu}{2} + \frac{\mu}{2} f(\xi v_i)\right) \exp( - \pi |\xi|^2 )\ d\xi.
\end{equation}
Here of course $d\xi$ is Lebesgue measure on the complex plane $\C$.
\end{lemma}

\begin{proof} From the Fourier identity
\begin{equation}\label{fourier-ident}
\exp(-\pi|z|^2) = \int_\C e( \Re( \xi z ) ) \exp(-\pi|\xi|^2)\ d\xi
\end{equation}
and \eqref{pumar} we have
\begin{equation}\label{pmuw}
\P_\mu(\v) = \int_\C \E e( \Re( \xi W_{\a^{(\mu)}}(\v) ) ) \exp(-\pi|\xi|^2)\ d\xi.
\end{equation}
On the other hand, from \eqref{wav}, \eqref{amu}, \eqref{fw} and independence we see that
$$ \E e( \Re( \xi W_{\a^{(\mu)}}(\v) ) ) = \prod_{i=1}^n (1-\frac{\mu}{2} + \frac{\mu}{2} f(\xi v_i))$$
and the claim follows.
\end{proof}

The relevance of concentration probability to the small ball probability is provided by the following lemma:

\begin{lemma}[Concentration probability bounds small ball probability]\label{concball}  For any tuple $\v$ and any $r > 0$, we have
$$ p_{r,\a}(\v) \le e^{\pi r^{2}} \P_1(\v).$$
\end{lemma}

In applications, $r$ will be very close to 0 and so the term $e^{\pi r^{2}} $ can be ignored.

\begin{proof} From Definition \ref{sbp}, it suffices to show that
$$ \P( W_\a(\v) \in B(z,r) ) \le e^{\pi r^{2}} \P_1(\v) $$
for any $z \in \C$.  Notice that
$$ \P( W_\a(\v) \in B(z,r) ) \le e^{\pi r^{2}}  \E \exp( - \pi |W_{\a}(\v) - z|^2 ).$$

Applying \eqref{fourier-ident} as in the proof of the preceding lemma, we have
$$ \E \exp( - \pi |W_{\a}(\v) - z|^2 )
= \int_\C \E e( \Re( \xi W_{\a}(\v) ) ) e( -\Re( \xi z ) ) \exp(-\pi|\xi|^2)\ d\xi.$$

The quantity $ |\E e( \Re( \xi W_{\a}(\v) ) )| $
can be expanded, using \eqref{wav} and  \eqref{fw}, as
$\prod_{i=1}^n f(\xi v_i)^{1/2}$. Since $f(\xi v_{i})^{1/2} \le \frac{1}{2} + \frac{1}{2} f(\xi v_{i})$, it follows that

$$ \left|\E e( \Re( \xi W_{\a}(\v) ) )\right| \leq \prod_{i=1}^n \left(\frac{1}{2} + \frac{1}{2} f(\xi v_i)\right).$$

The claim of the lemma follows from  the triangle inequality and Lemma \ref{pf2}.
\end{proof}

We now generalize the above lemma to the sparse case:

\begin{lemma}[Concentration probability bounds small ball probability, sparse version]\label{concball-sparse}  For any tuple $\v$ and any $r > 0, 1 \ge \mu >0$, we have
$$ p_{r,\a \I_{\mu}}(\v) \le e^{\pi r^{2}} \P_{\mu}(\v).$$
\end{lemma}

\begin{proof}
The proof is almost identical as the previous one. The only difference here is that we have  $ |\E e( \Re( \xi W_{\a \I_{\mu}}(\v) ) )| $ instead of  $ |\E e( \Re( \xi W_{\a}(\v) ) )| $.
Notice that $ |\E e( \Re( \xi W_{\a \I_{\mu}}(\v) ) )| $ can be expanded as  $\prod_{i=1}^n( (1-\mu) + \mu f(\xi v_i)^{1/2})$. Since $f(\xi v_{i})^{1/2} \le \frac{1}{2} + \frac{1}{2} f(\xi v_{i})$, it follows that
$$ |\E e( \Re( \xi W_{\a \I_{\mu}}(\v) ) )| \leq \prod_{i=1}^n \left((1- \frac{\mu}{2}) +\frac{\mu}{2}  f(\xi v_i)\right),$$

\noindent and again we are done using Lemma \ref{pf2}. \end{proof}

The concentration probability has several pleasant properties (cf. \cite[Lemma 5.1]{TVsing}):

\begin{lemma}[Properties of $\P_{\mu}$]\label{pumar-lemma} Let $0 < \mu \leq 1$.  Then the following properties hold.
\begin{itemize}
\item[(i)] The quantity $\P_{\mu}(\w)$ is monotone decreasing in $\mu$ and permutation invariant in $\w$.
\item[(ii)] For any tuples $\v, \w$ we have
$$ \P_{\mu}(\v \w) \leq \P_{\mu}(\v)$$
where $\v \w$ is the concatenation of $\v$ and $\w$.
\item[(iii)] For any tuples $\v, \w$ we have
$$ \P_{\mu}(\v) \P_{\mu}(\w) \leq 2 \P_{\mu}(\v \w).$$
\item[(iv)] For any $k \geq 1$ and tuple $\v$ we have
$$ \P_{\mu}(\v) \leq \P_{\mu/k}(\v^k)$$
where $\v^k$ is the concatenation of $k$ copies of $\v$.
\item[(v)] For any tuples $\v, \w_1, \ldots, \w_m$ we have
$$ \P_{\mu}(\v \w_1 \ldots \w_m ) \leq \left(\prod_{i=1}^m \P_\mu( \v \w_i^m ) \right)^{1/m}.$$
In particular, by the pigeonhole principle, there exists $1 \leq i \leq m$ such that
$$ \P_{\mu}(\v \w_1 \ldots \w_m ) \leq \P_\mu( \v \w_i^m ).$$
\end{itemize}
\end{lemma}

\begin{proof} Properties (i), (ii) are immediate from \eqref{pumar-fourier-2}.  To prove (iii), observe from \eqref{pumar} that
$$\P_{\mu}(\v) \P_{\mu}(\w) = \E \exp( - \pi (|W_{\a^{(\mu)}}(\v)|^2 +|W_{\a^{(\mu)}}(\w)|^2) )$$
where we require the walks $W_{\a^{(\mu)}}(\v), W_{\a^{(\mu)}}(\w)$ to be independent.  Using the arithmetic mean-geometric mean inequality
$$ |W_{\a^{(\mu)}}(\v)|^2 +|W_{\a^{(\mu)}}(\w)|^2 \geq \frac{1}{2} |W_{\a^{(\mu)}}(\v \w)|^2$$
followed by the Fourier identity
$$ \exp(-\pi|z|^2/2) =2 \int_\C e( \Re( \xi z ) ) \exp(-2\pi|\xi|^2)\ d\xi$$
we conclude that
$$\P_{\mu}(\v) \P_{\mu}(\w) \leq 2 \int_\C \E e( \Re( \xi W_{\a^{(\mu)}}(\v \w) ) ) \exp(-2\pi|\xi|^2)\ d\xi.$$
Comparing this with \eqref{pmuw} we obtain the claim.

The inequality (iv) follows easily from \eqref{pumar-fourier-2} and the elementary inequality $1-t \leq (1-t/k)^k$ for all $0 \leq t \leq 1$, which follows from the convexity of $\log(1-t)$.  Finally, the inequality (v) follows from \eqref{pumar-fourier-2} and H\"older's inequality.
\end{proof}

\section {The $\a$-norm of a complex number }\label{sec5}

In this section,
 we present a way to estimate
  the characteristic function $f$ (and hence the concentration probabilities $\P_\mu(\w)$) in terms of a more convenient expression.  Define the \emph{$\a$-norm} of a complex number $w \in \C$ by the formula
\begin{equation}\label{awdef}
 \| w \|_\a := \left(\E \| \Re( w (\a_1 - \a_2 ) ) \|_{\R/\Z}^2\right)^{1/2}
\end{equation}
where $\a_1, \a_2$ are iid copies of $\a$, and $\|t\|_{\R/\Z}$ denotes the distance from $t$ to the nearest integer.

\begin{example} If $\a$ is Bernoulli, then $\| w\|_{\a} =\frac{1}{\sqrt 2} \| \Re(2w) \| _{\R/\Z} $. So in this case the $\a$-norm of $w$ is basically the size of the fractional part of
$\Re(2w)$.
\end{example}

\begin{lemma}[Relationship between $f$ and $\a$-norm]  For any $w \in \C$ and $0 < \mu \leq 1$ we have
$$ (1 - \frac{\mu}{2}) + \frac{\mu}{2} f(w) \leq \exp\left( - \Omega( \mu \|w\|_\a^2 ) \right)$$
and thus by Lemma \ref{pf2} we have
\begin{equation}\label{pumar-fourier-3}
 \P_{\mu}(\w) \leq \int_\C \exp\left(-\Omega\left( \mu \sum_{i=1}^k \| \xi w_i \|_\a^2 \right)\right) \exp( - \pi |\xi|^2 )\ d\xi
\end{equation}
for any tuple $\w = (w_1,\ldots,w_k)$.
\end{lemma}

\begin{proof} In view of the elementary inequality $1 - t \leq \exp( - t )$ for $t \geq 0$, it will suffice to show that
$$ f(w) \leq 1 - \Omega( \|w\|_\a^2 ).$$
But from \eqref{fw} we have the identity
$$ f(w) = \Re \E e( \Re( w (\a_1 - \a_2) ) ) = \E \cos(2\pi \Re(w(\a_1-\a_2)))$$
and the claim follows from the elementary inequality $\cos(2\pi \theta) \leq 1- \Omega( \|\theta\|_{\R/\Z}^2 )$.
\end{proof}

We now record some useful properties of the $\a$-norm, which may help explain why we call it a ``norm'':

\begin{lemma}[Properties of $\a$-norm]\label{panorm}
The following properties hold:

\begin{itemize}
\item[(i)] For any $w \in \C$,  $0 \leq \| w \|_\a \leq 1$ and $\| -w \|_\a = \|w \|_\a$.
\item[(ii)] For any $z, w \in \C$, $\| z+w \|_\a \leq \|z\|_\a + \|w\|_\a$.
\item[(iii)] If $\a$ has $\kappa$-controlled second moment for some positive constant $\kappa$, then
there exists a positive constant $c$ depending on $\kappa$ such that
$\| z \|_\a \gg |\Re(z)|$ for all $z \in B(0,c)$.
\end{itemize}
\end{lemma}

\begin{proof} Property (i) is obvious.  Property (ii) follows from the triangle inequality for $L^2$
and the elementary observation that $\|x+y\|_{\R/\Z} \leq \|x\|_{\R/\Z} + \|y\|_{\R/\Z}$.

Now we prove (iii).  Let $z \in B(0,c)$ for some small $c$.  From \eqref{awdef} it suffices to show that
$$ \E \| \Re( z (\a_1 - \a_2 ) ) \|_{\R/\Z}^2 \gg |\Re(z)|^2.$$
On the other hand, from \eqref{eyot} we have
$$ \E |\Re( \a )|^2 \I( |\a| \leq K ) \geq \frac{1}{K}$$
for some $K = O(1)$.  In particular $\P( |\a| \leq K ) \gg 1$.  So if we let $\b_i$ for $i=1,2$ be $\a_i$ conditioned on the event $|\a_i| \leq K$, it suffices to show that
$$ \E \| \Re( z (\b_1 - \b_2 ) ) \|_{\R/\Z}^2 \gg |\Re(z)|^2.$$
If $c$ is small enough depending on $K$, then $|z(\b_1-\b_2)| \leq \frac{1}{2}$, so it suffices to show that
$$ \E  |\Re( z (\b_1 - \b_2 ) )| \gg |\Re(z)|^2.$$
But this follows by conditioning on $\b_2$ and then using \eqref{eyot}.
\end{proof}

\section{Generalized arithmetic progressions and the forward Littlewood-Offord theorem}\label{sec6}

As in previous literature, our Littlewood-Offord theorems shall involve \emph{generalized arithmetic progressions} (GAPs), which we now define.

\begin{definition}[Generalized arithmetic progression]  If $v_1,\ldots,v_r$ are complex numbers and
$L_1,\ldots,L_r$ are positive numbers, we define the
\emph{symmetric generalized arithmetic progression} (or \emph{symmetric GAP} for short)
$$Q = \GAP((v_1,\ldots,v_r),(L_1,\ldots,L_r))$$ to be the set
$$ Q := \{ n_1 v_1 + \ldots + n_r v_r| n_1,\ldots,n_r \in \Z;
|n_i| \leq L_i \hbox{ for all } i \} \subset \C.$$
We refer to $r$ as the \emph{rank} of $Q$, $v_1,\ldots,v_r$ as the \emph{generators}, and $L_1,\ldots, L_r$ as the \emph{dimensions}.

If all the sums $n_1 v_1 + \ldots + n_r v_r$ are distinct, we say that $Q$ is \emph{proper}.  For $t > 0$, we define the dilate $tQ$ of $Q$ as
$$ tQ :=\GAP((v_1,\ldots,v_r),(tL_1,\ldots,tL_r)).$$
Finally, if $L_1=\ldots=L_r=L$, we abbreviate $\GAP((v_1,\ldots,v_r),(L,\ldots,L))$ as $\GAP((v_1,\ldots,v_r),L)$.
\end{definition}

GAPs are a fundamental object in additive combinatorics and they have played a crucial role in our earlier papers on
Inverse Littlewood-Offord theorems and least singular values \cite{TVsing, TVstoc}.
For a  detailed discussion about these objects, we refer to \cite{TVbook}.

\begin{remark}
It is  helpful to view $Q$ as the image of the integral box
$$\{(n_{1}, \cdots, n_{r}) |  |n_{i}| \le L_{i} , 1\le i \le r \} \subset \Z^{r}$$ under the linear map
$\Phi$ that sends the point $(n_{1}, \dots, n_{r})$ to $n_{1}v_{1} + \dots  + n_{r}v_{r} $.  $Q$ is proper if $\Phi$ is one-to-one.
\end{remark}

We use the following two simple lemmas frequently:

\begin{lemma}[Doubling property] Let $Q$ be a symmetric GAP of rank $r$ and $t \ge 1$. Then

$$\#(tQ) \ll_r t^{r} \# Q. $$ \end{lemma}

\begin{proof} One can cover $tQ$ by $O(t^r)$ translates of $Q$.
\end{proof}

\begin{lemma}[Pigeonhole principle]\label{pig}  Let $Q \subset \C$ be a finite set, and let $\Omega \subset \C$ be a set which can be covered by at most $M$ balls of radius $r/2$.  Then we have
$$ \#\left( (Q-Q) \cap B(0,r) \right ) \geq \frac{\#(Q \cap \Omega)}{M}.$$
\end{lemma}

\begin{proof} We can of course assume that $Q \cap \Omega$ is non-empty.  By the pigeonhole principle, we can find a ball $B(z,r/2)$ of radius $r/2$ which contains at least $\# (Q \cap \Omega)/M$ elements of $Q \cap \Omega$; in particular it contains at least one element $z_0$ of $Q$.  Since $(Q \cap \Omega \cap B(z,r/2))-z_0$ is contained in $(Q-Q) \cap B(0,r)$, the claim follows.
\end{proof}

For a  GAP $Q=\GAP((v_{1}, \dots, v_{r}), ( L_{1}, \dots, L_{r}))$,
define the \emph{dispersion} $ \D(Q)$ to be the quantity
\begin{equation}\label{dw}
 \D(Q) := \frac{\# Q}{\# (Q\cap B(0,1)) }.
\end{equation}

\begin{remark} The quantity $\D(Q)$ is very close to the metric entropy $\N_1(Q)$ of $Q$, indeed
 simple volume packing arguments (cf. Lemma \ref{pig}) show that $\D(Q) =\Theta_r( \N_1(Q))$.  We will however not use that fact here.
\end{remark}

This quantity turns out to control the concentration probability of
certain random walks associated with $Q$:

\begin{theorem}[Forward Littlewood-Offord theorem]\label{flot}
For any $0 < \mu \leq 1$, $\eps > 0$, and complex numbers $v_{1}, \dots, v_{r}$,  we have
$$ \P_\mu( v_1^{L_1^2} \ldots v_r^{L_r^2} ) \ll_{\eps,r}  \D(Q)^{-1+\eps},$$
\noindent where
$$Q = \GAP((v_1, \ldots, v_r), ( \sqrt {\mu}L_1, \ldots,  \sqrt{\mu} L_r) ). $$
\end{theorem}

This ``forward Littlewood-Offord theorem'' will be crucial in establishing Theorem \ref{ilo}.  To give the reader some feeling about this estimate,  let us first consider a toy case when
$\a$ is Bernoulli and $\mu=1$. The adjusted random variable $\a^{(\mu)}$ equals $0$ with probability $3/4$ and $\pm 2$ with probability $1/8$.

Assume furthermore that the $v_{i}$ are non-zero integers and $Q$ is proper.
 Thus $Q \cap B(0,1) \subset \{-1,0,1\}$ and the desired bound becomes
$$ \P_1( v_1^{L_1^2} \ldots v_r^{L_r^2} ) \ll_{\eps, r}  (\# Q)^{-1+\eps}.$$

Consider a (lazy) random walk $W$ starting at $0$. At step $j$, stay with probability $1/2$ and move to right or left by an amount
$v_{j}$ with probability $1/8$. The terminal point after $n$ step is exactly the random variable
$$ W_{\a^{(1)}} (v_{1}\dots v_{n} )= \a^{(1)}_{1} v_{1} + \dots + \a^{(1)}_{n} v_{n} . $$

Since $\P_{\mu} (\v) := \E \exp (-\pi | W_{\a^{(\mu)}} (\v)|^{2}) $, the quantity
 $ \P_1( v_1^{L_1^2} \ldots v_r^{L_r^2} )$ can be bounded from above by the sum of the probability that the lazy random walk with
 $L_{1}^{2}$ steps of size $v_{1}$,  \dots, $L_{r}^{2}$ steps of size $v_{r}$ ends up on a point with absolute value at most
 $10 \log (\#Q)$ and a negligible term which is much smaller than $( \#Q)^{-1}$.

 Notice that the coefficient of $v_{j}$ is the sum of $L_{j}^{2}$ iid copies of $\a^{(1)}$.  It is well known that the distribution of this sum is roughly
 uniform on the interval $[-L_{j}, L_{j}]$. (By roughly uniform,
 we mean that for any two integers in this interval, the ratio of their masses is bounded from above by a positive constant.)
 Thus, the main observation here
 (and somehow the essence of the theorem)  is that the  end point of the walk (conditioned on the fact that the coefficient of $v_{j}$ belongs to $[-L_{j}, L_{j}]$ ) is roughly uniform
 in $Q$.
  It follows that  the probability that it has absolute value $O(\log \# Q)$  can be bounded from above by
 $O(\frac{\log \#Q}{\# Q}) \le \#Q^{-1+\eps} $, giving the desired bound.

 This argument can be made rigorous for random variables $\a$ with discrete supports, even when $Q$ is not proper. However, the proof for the general case is more complicated.  The main technical tool needed is the following level set estimate:

\begin{lemma}[Level set estimate]\label{lemma:levelset}
Given a GAP $Q=\GAP(v_{1}, \dots, v_{r}, L_{1}, \dots, L_{r})$, a complex number $\xi_{0}$, and $\eps > 0$, let  $\Sigma \subset \C$ be the set
\begin{equation}\label{omdef}
\Sigma := \left\{ \xi \in B(\xi_0,1) \bigl| \| \xi v_i \|_\a \le \D(Q)^\eps / L_i \hbox{ for all } 1 \leq i \leq r \right\}.
\end{equation}
Then
\begin{equation}\label{omeq}
 \mes(\Sigma) \ll_{\eps,r} \D(Q)^{-1+O_r(\eps)}.
\end{equation}
\end{lemma}

We will prove Lemma \ref{lemma:levelset} in Sections \ref{lacunary-set}-\ref{sec7} below.
 For now, let us show how it implies Theorem \ref{flot}.

\begin{proof}[Proof of Theorem \ref{flot} assuming Lemma \ref{lemma:levelset}]
We abbreviate $\D := \D(Q)$.  In view of \eqref{pumar-fourier-3}, it suffices to show that
$$ \int_\C \exp(-\Omega( \mu \sum_{i=1}^r L_i^2 \| \xi v_i \|_\a^2 )) \exp( - \pi |\xi|^2 )\ d\xi \ll_{r, \eps} \D^{-1+\eps}.$$

Covering $\C$ by balls of radius $1$, it thus suffices to show that
$$ \int_{B(\xi_0,1)} \exp(-\Omega( \sum_{i=1}^r \mu L_i^2 \| \xi v_i \|_\a^2 ))\ d\xi \ll_{r,\eps} \D^{-1+\eps}$$
for all $\xi_0 \in \C$.

Now we fix $\xi_{0}$. Let $c$ be a small positive constant to be determined. It is clear that if $\D$ is sufficiently large, then
$$ \int_{B(\xi_0,1)} \exp\left(-\Omega( \sum_{i=1}^r \mu L_i^2 \| \xi v_i \|_\a^2 )\right)\ d\xi  \le \mes (\Sigma) + \D^{-1} $$
\noindent where
$$\Sigma := \{ \xi \in B(\xi_0,1)| \| \xi v_i \|_\a \le \frac{ \D^{c\eps}  }{\sqrt \mu  L_i } \hbox{ for all } 1 \leq i \leq r \}. $$

(In fact,  $\D^{c\eps} $ can be replaced by $C \log \D$ for some large constant $C$.) By Lemma \ref{lemma:levelset},
$$\mes (\Sigma) \le \D^{-1 +O(c\eps)}. $$

We choose $c$ equal half of the reciprocal of the hidden constant in $O$. It follows that
$$\mes (\Sigma) \le \D^{-1+\eps/2}, $$
\noindent which implies
$$ \int_{B(\xi_0,1)} \exp(-\Omega( \sum_{i=1}^r \mu L_i^2 \| \xi v_i \|_\a^2 ))\ d\xi  \le \D^{-1+\eps/2} + \D^{-1}  \le \D^{-1+\eps},$$
\noindent concluding the proof.
\end{proof}

To conclude the proof of Theorem \ref{flot}, we need to establish Lemma \ref{lemma:levelset}.  This is the purpose of the next two sections.

\section{Lacunary sets inside GAPs}\label{lacunary-set}

Let $S$ be a set. We shall informally call a sequence $w_{1}, \dots,
w_{d}$ of elements of $S$ {\it lacunary} if the ratio
$\frac{|w_{i-1}|}{|w_{i}|}$ is large for all $1 < i \le d$. The goal
of this section is to show that a large GAP always contains
  a large lacunary subset with some prescribed properties.  This fact will be a key ingredient in the proof of Lemma \ref{lemma:levelset} (and hence Theorem \ref{flot}), which we give in the next section.

  To give the reader some motivation, let us consider the toy case when $Q$ is an interval, say $\{-s,-s+1, \dots, s-1, s\}$. Given a ratio $K>2$ and a constant $R>1$ (say), we can
  easily   find $d$ elements
  $w_{1}, \dots, w_{d}$ such that $|w_{d}| \ge R$ and $\frac{|w_{i}|}{|w_{i+1} |} \ge K$ where $d$ satisfies
  $$ \#Q \ll  K^{d} R. $$

The main result of this section is a generalization of the above observation for general GAPs.

\begin{lemma}[Lacunarity lemma]\label{lac}  Let $K \geq 1$, let $Q$ be a symmetric GAP of rank $r$,
and let $R \geq 0$ be a radius.  Then there exists, for some $d \geq
0$, ``primary vectors'' $w_1,\ldots,w_d \in Q$, and ``secondary
vectors'' $w'_1,\ldots,w'_d \in Q$ with the following properties:
\begin{itemize}
\item[(i)] (Lacunarity) We have $|w_i| \geq K |w_{i+1}|$ for all $1 \leq i \leq d-1$.
\item[(ii)] (Secondary bounds) We have $|w_i| > R$ and $|w'_i| \leq |w_i|$ for all $1 \leq i \leq d$.
\item[(iii)]  (Many vectors) We have
\begin{equation}\label{qork}
 \# Q \leq \left[\prod_{i=1}^d O_r(K K_i)\right] \#(Q \cap B(0,R))
\end{equation}
where $1 \leq K_i \leq 1+K$ is the quantity
\begin{equation}\label{ki}
 K_i := 1 + K \left|\Im( \frac{w'_i}{w_i} )\right|.
\end{equation}
\item[(iv)] (Crude upper bound) We have
\begin{equation}\label{dr}
d \ll_r 1 + \frac{\log \frac{\# Q}{\# (Q \cap B(0,R))}}{\log K}.
\end{equation}
\end{itemize}
\end{lemma}

\begin{remark} The secondary vectors are necessary here because $Q$ is taking values in the complex numbers; if $Q \subset \R$ then we could simply take $w'_i = 0$ (and thus $K_i=1$) for all $i$.  The reader may wish to follow the argument below in the real case (and for $R=0$), as it is somewhat simpler in that case.  The bound \eqref{qork} may seem strange, but it is best possible except for the $O_r(\cdot)$ factor, and we will need such a tight estimate in our applications.  The vectors $w_1,\ldots,w_d, w'_1,\ldots,w'_d$ are somewhat analogous to the \emph{Minkowski basis} of a lattice with respect to a convex body, thus \eqref{qork} can be viewed as a variant of Minkowski's second theorem.
\end{remark}

\begin{proof}  By increasing $K$ if necessary we may assume $K$ to be larger than any given constant depending on $r$.  We can also assume that $Q$ is not contained in $B(0,R)$, as the claim is obvious otherwise.

We perform the following algorithm.  We set $d_0 := C_r \left(1 + \frac{\log \frac{\# Q}{\#(Q \cap B(0,R))}}{\log K}\right)$ for some sufficiently large constant $C_r$ depending only on $r$.
\begin{itemize}
\item[Step 0] Initialize $i = 1$.  We also adopt the convention that $w_0=\infty$.
\item[Step 1] Let $Q_i := 2^{-d_0+i} Q \cap B(0,|w_{i-1}|/K)$.  If $Q_i \subset B(0,R)$ then set $d := i-1$ and \textbf{STOP}.  Otherwise, let $w_i \in Q_i$ be chosen such that $|w_i|$ is maximal; thus $|w_i| \leq |w_{i-1}|/K$, $|w_i| > R$ and $Q \subset B(0,|w_i|)$.
\item[Step 2] Let $w'_i \in Q_i$ be chosen to maximize the quantity $K_i$ defined in \eqref{ki}.  Observe that $|w'_i| \leq |w_i|$.
\item[Step 3] From elementary complex geometry we see that $Q_i$ is now contained in a rectangle of dimensions $O(|w_i|) \times O( \frac{K_i}{K} |w_i| )$.  This rectangle can be covered by $O( K K_i )$ disks of radius $|w_i|/2K$.  Applying Lemma \ref{pig}, we conclude that the set
$$ Q_{i+1} := 2^{-d_0+i+1} Q \cap B(0,|w_i|/K) \supset (Q_i-Q_i) \cap B(0,|w_i|/K)$$
obeys the lower bound
\begin{equation}\label{qinc}
 \# Q_{i+1} \gg \frac{1}{K K_i} \# Q_i.
 \end{equation}
\item[Step 4] Increment $i$ to $i+1$ and return to Step 1.
\end{itemize}
Since $w_1, w_2, \ldots$ have decreasing magnitude and lie in the finite set $Q$ we see that this algorithm terminates in finite time.  In fact we claim that this algorithm terminates before step $d_0$.  For if the algorithm reaches stage $d_0$, we have obtained $w_1, \ldots,w_{d_0} \in Q$ obeying the lacunarity condition $|w_i| \leq |w_{i-1}|/K$.  This implies that the GAP $\GAP( (w_1, \ldots, w_{d_0}), K/10 )$ is proper, and that the pairwise sums between $\GAP((w_1, \ldots, w_{d_0}), K/10)$ and $2Q \cap B(0,R/10)$ are distinct and contained in $(d_0 K+1) 2Q$.  But this implies that
$$ (K/10)^{d_0} \#(2Q \cap B(0,R/10)) \ll \#( d_0 K Q ) \leq O(d_0 K)^r \# Q.$$
Also, since $B(0,R)$ can be covered by $O(1)$ balls of radius $R/20$, we see from Lemma \ref{pig} that
$$ \#( 2Q \cap B(0,R/10) ) \gg \# (Q \cap B(0,R))$$
and thus
$$ \Omega(K)^{d_0-r} \ll O(d_0)^r \frac{\# Q}{\# (Q \cap B(0,R))}.$$
But from definition of $d_0$, we see that this is impossible if $C_r$ is chosen sufficiently large (recall we are taking $K$ large compared to $r$).  Thus we have $d \leq d_0$, which in particular implies that $w_1,\ldots,w_d$ and $w'_1,\ldots,w'_d$ lie in $Q$.  Since $Q$ is not contained in $B(0,R)$ we also have $d \geq 1$.

Next, we observe from \eqref{qinc} that
$$ \#(Q \cap B(0,R)) \geq \# Q_{d+1} \geq \left(\prod_{i=1}^d \Omega(\frac{1}{KK_i})\right) \# Q_1.$$
Now we can cover $Q$ by $O_r(1)^{d_0}$ copies of $Q_1 = 2^{-d_0+1} Q$, and thus
\begin{equation}\label{qork-2}
\# Q \leq \left[\prod_{i=1}^d O(K K_i)\right] O_r(1)^{d_0} \#( Q \cap B(0,R)).
\end{equation}
In particular, since $K+1 \geq K_1$ and $d_0 \geq d$ we have
$$ \frac{\# Q}{\# (Q \cap B(0,R))} \leq (K+1)^{2d} O_r(1)^{d_0};$$
using the definition of $d_0$ and recalling that $K$ is large
compared to $r$ we conclude that $d \gg_r d_0$.  The claim
\eqref{qork} now follows from \eqref{qork-2}. The remaining claims
are easily verified from the construction.
\end{proof}

\section{Proof of Lemma \ref{lemma:levelset}}\label{sec7}

We are now ready to prove Lemma \ref{lemma:levelset}.
In the following, $Q$ is fixed and we write $\D$ instead of $\D(Q)$.  We also fix $r$ and allow all implied constants to depend on $r$.  We may assume without loss of generality that $\D$ is large compared with $\eps$, since the claim is trivial otherwise.

Let $K := \D^{\eps}$; since $\D$ is assumed large compared to $\eps$, we see that $K$ is also.  We apply Lemma \ref{lac} (with $R=1$, and to the GAP $\frac{1}{K^4} Q$) to obtain vectors
\begin{equation}\label{kq}
w_1, \ldots, w_d, w'_1,\ldots,w'_d \in \frac{1}{K^4} Q
\end{equation}
for some $d = O(1/\eps)$ such that $|w_i| \geq K |w_{i+1}|$ for all $1 \leq i \leq d-1$,  $|w_i| > 1$ and $|w'_i| \leq |w_i|$ for all $1 \leq i \leq d$, and
$$ \# \left(\frac{1}{K^4} Q\right) \leq \left[\prod_{i=1}^d O(K K_i)\right] \#(Q \cap B(0,1))$$
where $K_i$ is defined in \eqref{ki}.  Since $Q$ has rank $O(1)$, we have
$$ \# \left(\frac{1}{K^4} Q\right) \gg K^{-O(1)} \# Q = \D^{1-O(\eps)} \#(Q \cap B(0,1))$$
and thus (since $d = O(1/\eps)$ and $\D$ is large compared with $\eps$)
\begin{equation}\label{pidk}
 \prod_{i=1}^d K K_i \geq \D^{1-O(\eps)}.
\end{equation}
From \eqref{omdef}, Lemma \ref{panorm}, and \eqref{kq} we see that
\begin{equation}\label{xiwi}
\| \xi w_i \|_\a, \|\xi w'_i \|_\a \ll \frac{1}{K^3}
\end{equation}
for all $1 \leq i \leq d$ and $\xi \in \Sigma$.

For $1 \leq i \leq d$, define $\zeta_i := \frac{1}{K^2 w_i}$ and $\zeta'_i := \sqrt{-1} \frac{1}{K K_i w_i}$.  Let $P$ be the GAP
$$ P := \GAP\left( (\zeta_1, \ldots, \zeta_d), \frac{K}{100} \right) + \GAP\left( (\zeta'_1, \ldots, \zeta'_d), (\frac{K_1}{100}, \ldots, \frac{K_d}{100}) \right).$$
One should view $P$ as a kind of ``dual'' to $Q$.  It has the following properties:

\begin{lemma}[Properties of $P$] We have
\begin{itemize}
\item[(i)] $P$ is proper.
\item[(ii)] $\# P \geq \D^{1-O(\eps)}$.
\item[(iii)] $P \subset B( 0, O(1/K) )$.
\item[(iv)] If $z, z' \in P$ are distinct, then $z+\Sigma$ and $z'+\Sigma$ are disjoint.
\end{itemize}
\end{lemma}

\begin{proof} We first verify (i).  If $P$ is not proper, then we have a linear relation
$$ n_1 \zeta_1 + \ldots + n_d \zeta_d + m_1 \zeta'_1 + \ldots + m_d \zeta'_d = 0$$
for some integers $n_1,\ldots,n_r,m_1,\ldots,m_r$, not all zero, with $|n_i| \leq K/50$ and $|m_i| \leq K_i/50$ for $1 \leq i \leq r$.  Let $j$ be the largest index such that $(n_j, m_j)$ is non-zero.  If $1 \leq i < j$, then from the properties of $w_i$ we have
$$ |w_i| \ge K^{j-i} |w_j|$$
and so
$$|\zeta_i| \leq \frac{|\zeta_j|}{K^{j-i}}; \quad |\zeta'_i| \leq \frac{|\zeta_j|}{K^{j-i-1} K_i}.$$
From the triangle inequality we then have
$$ |n_1 \zeta_1 + \ldots + n_{j-1} \zeta_{j-1} + m_1 \zeta'_1 + \ldots + m_{j-1} \zeta'_{j-1}| \leq \frac{|\zeta_j|}{10}$$
and thus
$$ |n_j \zeta_j + m_j \zeta'_j| \leq \frac{|\zeta_j|}{10}.$$
On the other hand, since $n_j,m_j$ are integers which are not both zero, and $\zeta'_j = \sqrt{-1} \frac{K}{K_j} \zeta_j$, and $K/K_j \geq 1/2$, we see that
$$ |n_j \zeta_j + m_j \zeta'_j| \geq \frac{|\zeta_j|}{2},$$
a contradiction.

From (i) we also see that
$$ \# P \geq \prod_{i=1}^d \Omega(\frac{K}{100}) \Omega(\frac{K_i}{100})$$
and so (ii) now follows from \eqref{pidk} (recalling that $d=O(1/\eps)$ and $\D$ is large compared to $\eps$).

Now we prove (iii).  If $z \in P$, then we see from the triangle inequality that
$$ |z| \leq \frac{1}{100} \sum_{j=1}^d K |\zeta_j| + K_j |\zeta'_j| \leq \frac{1}{50 K} \sum_{j=1}^d \frac{1}{|w_j|} \leq \frac{1}{10 K |w_d|}$$
by lacunarity.  But by construction $|w_d| \geq 1$, and the claim follows.

Now we prove (iv).  If the claim was false, then we could find distinct $z, z' \in P$ and $\xi,\xi' \in \Sigma$ such that $z-z' = \xi-\xi'$.  We can then write
$$ z-z' = n_1 \zeta_1 + \ldots + n_d \zeta_d + m_1 \zeta'_1 + \ldots + m_d \zeta'_d = \xi - \xi'$$
for some integers $n_1,\ldots,n_r,m_1,\ldots,m_r$, not all zero, with $|n_i| \leq K/50$ and $|m_i| \leq K_i/50$.

Let $j$ be the largest index such that $(n_j, m_j)$ is non-zero.  From \eqref{xiwi} and Lemma \ref{panorm} we have
\begin{equation}\label{zzw}
\| (z-z') w_j \|_\a, \|(z-z') w'_j \|_\a \ll \frac{1}{K^3}.
\end{equation}
On the other hand, from the triangle inequality we have
$$ |z-z'| \leq \frac{1}{100} \sum_{i=1}^j K |\zeta_i| + K_i |\zeta'_i| \leq \frac{1}{50K} \sum_{i=1}^j \frac{1}{|w_i|} \leq \frac{1}{10 K|w_j|}$$
by lacunarity, and thus
$$|(z-z') w'_j| \leq |(z-z') w_j| \ll \frac{1}{K}.$$
If $K$ is large enough, then we can apply Lemma \ref{panorm} to conclude from \eqref{zzw} that
$$
\Re((z-z') w_j), \Re((z-z') w'_j) = O(\frac{1}{K^3}).$$
On the other hand, observe that
$$ |z-z' - (n_j \zeta_j + m_j \zeta'_j)| \leq \frac{1}{100} \sum_{i=1}^{j-1} K |\zeta_i| + K_i |\zeta'_i| \ll \frac{1}{50K} \sum_{i=1}^{j-1} \frac{1}{|w_i|} \ll \frac{1}{10 K^2|w_j|}$$
and so by the triangle inequality
$$
|\Re((n_j \zeta_j + m_j \zeta'_j) w_j)|, |\Re((n_j \zeta_j + m_j \zeta'_j) w'_j)| \leq \frac{1}{5K^2}$$
if $K$ is large enough.  On the other hand, by construction of $\zeta_j, \zeta'_j$ we have
$$ \Re((n_j \zeta_j + m_j \zeta'_j) w_j) = \frac{n_j}{K^2}.$$
Since $n_j$ is an integer, we conclude $n_j=0$.  In that case we have
$$ |\Re((n_j \zeta_j + m_j \zeta'_j) w'_j)| = \frac{|m_j|}{K K_j} |\Im( \frac{w'_j}{w_j} )| \geq \frac{|m_j|}{2K^2}$$
if $K_j \geq 2$, by construction of $\zeta'_j$ and $K_j$.  Since $m_j$ is an integer,
we conclude $m_j=0$.  On the other hand, if $K_j < 2$, then we have $m_j=0$ as well, since $|m_j| \leq K_j/50$.  But $(n_j,m_j)$ is non-zero, a contradiction.
\end{proof}

From properties (ii), (iii), (iv) we see that
$$ \mes(B(0,O(1))) \geq \D^{1-O(\eps)} \mes(\Sigma)$$
and the claim \eqref{omeq} follows.

\section{Structure of weak elements}\label{sec9}

 Let $Q$ be a GAP.  Extend $Q$ by a new dimension generated by a new element $z$; $Q'= Q+ \{-kz, \cdots, kz\}$. We call $z$ {\it weak} if
$\#Q'$ is only slightly more than $\#Q$. The goal of this section is to quantify (and generalize) the following phenomenon:

\vskip2mm

\centerline {\it The set of weak $z$ has small entropy. }

The reader may find the following simple example illustrative.
Assume that $Q$ is the interval $[-s, \dots, s]$. Assume that $Q':=
Q+ \{-kz, \cdots, kz\}$ has cardinality at most $ls$, where
$l=k^{\delta}$ for some small positive $\delta$.

Consider the interval $Q_1 := \{x \in \Z| |x| \le sl/k \}$. The sets $x+ \{z, \cdots, kz \}, x\in Q_{1}$ are subsets of $Q'$. Since $\#Q_{1} > ls/k$, these sets are not disjoint. Thus, we have
$x+jz =x'+j'z$ for some distinct $x, x' \in Q_{1}$ and $1\le j\neq j' \le k$. This implies that
$$z \in \bigcup_{1\le \tau \le k}  \frac{1}{\tau} \cdot (Q_{1}-Q_1). $$

This already gives a bound $k\#(Q_1-Q_1)=O(l\#Q)= O(ls)$ on the cardinality of the possible $z$. But  this bound can be improved further (this improvement is critical later on).
Consider the set $x+ \{0, \cdots, lz\}$ with $x \in Q$. By the same argument as before, these sets are not disjoint, and we can conclude that
$$z \in \bigcup_{1\le \tau' \le l} \frac{1}{\tau'} \cdot (Q-Q). $$

Thus, $z$ has two representations
$$z = \frac{x}{\tau} = \frac{x'}{\tau'} $$
\noindent for $x\in Q_{1}-Q_1, 1\le \tau \le k$ and $x' \in Q-Q, 1\le \tau'\le l$.  If $\frac{x}{\tau}$ is irreducible, then $\tau \le l$ and the number of $z$'s of this form is only at most
$l \#(Q_{1}-Q_1)= O(\frac{l^{2}}{k} s)$. If it is not, then $\operatorname{gcd}(x, \tau) \ge  \frac{\tau}{l}$. The number of $x$ satisfying this condition in $Q_{1}-Q_1$ is at most
$O(\frac{l}{\tau} \# Q_{1})$. Thus, the number of $z$'s  is at most
$\sum_{\tau=l}^{k} O( \frac{l}{\tau} \# Q_{1} ) = O(\frac{l^{2}}{k} s)$, using the bound on $\#Q_{1}$ and the fact that $l= k^{\Omega(1)}$. Thus, altogether we obtain the bound
$$ O\left(\frac{l^{2}} {k} s\right)= O\left(\frac{l^{2}}{k}\#Q \right)$$
\noindent which is much better than the previous bound $O(l \#Q)$. The term $k^{-1}$ will play a critical role in later proofs.

The main result of this section is
a generalization of this very special case.

 \begin{lemma} \label{lemma:weakz} Let $w_1,\ldots,w_r$ be complex numbers and $Q=\GAP((w_{1}, \dots, w_{r}), (L_{1}, \dots, L_{r}))$. Let $z$ be a complex number and $k$ a positive integer. Define
$$Q':= Q + \GAP(z,k)= Q+ \{-kz, \dots, kz\}. $$
Let $Z$ denote the set of all complex numbers $z$ such that
$$ \D( Q' ) < l \D(Q).$$
Then $Z$ has a $24$-net of size at most  $1 + O_{r}(l^{4}k^{-1} \D(Q))$.
\end{lemma}

\begin{remark}
The  $24$-net  can be replaced by  an $1$-net  if we replace the bound $1 + O_{r}(l^{4}k^{-1} \D(Q))$ by
$ O_{r}(1+ l^{4}k^{-1} \D(Q))$.  However, it is important to us to have the current formulation, as in the case when $l^{4}k^{-1} \D(Q)=o(1)$
the net will have size exactly $1$.  The power of $l^4$ might be improvable, but we will not need this improvement here, as $l$ will always be relatively small for us compared to other parameters such as $k$ and $\D(Q)$.
\end{remark}

\begin{proof} Let $z \in Z$. By definition of $Z$, we have
$$ \frac{\#(Q + \GAP(z,k))}{\#((Q + \GAP(z,k)) \cap B(0,1))} \le l \frac{\# Q}{\#(Q \cap B(0,1))}.$$
Let $W \subset \frac{1}{2} Q$ be a maximal $1$-net of $\frac{1}{2} Q$, then we see that the sets $w + (Q \cap B(0,1))$ for $w \in W$ cover $\frac{1}{2} Q$, and thus
$$ \# W \geq \frac{\# (\frac{1}{2} Q)}{\#(Q \cap B(0,1))} \gg_r \frac{\# Q}{\#(Q \cap B(0,1))},$$
\noindent thanks to the easily verified fact that $\# (\frac{1}{2} Q) \gg_{r } \#Q$.

Refine the $1$-net $W$ to a maximal $2$-net $W'$. We have $\# W' \gg_{r }\# W$ and thus
\begin{equation}\label{quzak}
 \#(Q + \GAP(z,k)) \ll_r l  \#((Q + \GAP(z,k)) \cap B(0,1)) \# W'.
\end{equation}

Now, define the set
$$ L := \{ -2k \leq j \leq 2k | jz \in 2Q + B(0,2) \}.$$

A simple greedy algorithm argument (using the symmetry of $L$) shows that we can find a set $J \subset \{-k,\ldots,k\}$ of cardinality $\# J \gg \frac{k}{\# L}$ such that $j_1 - j_2 \not \in L$ for any distinct $j_1, j_2 \in J$.  Consider the sets
$jz + W' + ((Q + \GAP(z,k)) \cap B(0,1))$ for $j \in J$. By the construction, we can verify that

\begin{itemize}
\item[(a)] These sets are disjoint (thanks to the definition of $J$ and $L$).
\item[(b)] Every set lies in $2(Q + \GAP(z,k))$ (since $|j| \le k$ and $W' \subset \frac{1}{2}Q). $
\item[(c)] Each set has  cardinality $(\# W') \#((Q + \GAP(z,k)) \cap B(0,1))$ (since $W'$ is a $2$-net).
\end{itemize}

It follows that
$$ \#J (\# W') \# ((Q + \GAP(z,k)) \cap B(0,1)) \ll \#(2(Q + \GAP(z,k))) \ll_r \#(Q + \GAP(z,k)).$$
Combining this with \eqref{quzak} we conclude that
$$ \# J \ll_{r} l. $$
On the other hand, $\#J \gg \frac{k}{\# L}$, so
\begin{equation}\label{lag}
\# L \gg_r l^{-1}k,
\end{equation}
\noindent  which asserts that many multiples of $z$ are close to $2Q$.

Let $R_0$ be the smallest radius such that
\begin{equation}\label{qbr}
 \#\left(10Q \cap B(0,R_0)\right) \geq C_{r} l k^{-1} \# Q,
\end{equation}
\noindent where $C_{r} $ is a sufficiently large constant depending on $r$.

By definition,
\begin{equation}\label{qbr-2}
 \#\left(10Q \cap B(0,R_0/2)\right) =O_{r} (lk^{-1} \# Q).
\end{equation}
Assume, for a moment,  that $|z| \geq 2R_{0} +4$.  By the definition of $L$, we can find, for each $j \in L$, an element  $\zeta_j \in 2Q$ such that $|jz-\zeta_j| \leq 2$.
(If there are many $\zeta_{j}$, fix one arbitrarily.) Let $j$ and $j'$ be two different indices, then
$$|\zeta_{j}- \zeta_{j'}| \ge |(j-j')| |z| -4 \ge |z| -4. $$
This implies that the  sets $\zeta_j + (10Q \cap B(0,R_0))$ are disjoint. Furthermore, as $\zeta_{j} \in 2Q$, they all  lie in $12Q$. Therefore,
$$ (\# L) \#\left(10Q \cap B(0,R_0)\right) \leq \#(12Q) \ll_r \# Q.$$
But this contradicts \eqref{qbr} if we choose $C_r$ sufficiently large. Thus we have
$$ |z| < 2R_0+4.$$

If $R_0 <10$, then  $z < 24$ and  $Z$ has a maximal $24$-net of cardinality $1$ and we are done.

From now on, we assume $R_0 \geq 10$.  Thus $|z| < 3R_0$.

From \eqref{lag} and the pigeonhole principle we can find $j, j' \in L$ such that $0 < |j-j'| \ll_r l $.  Thus there exists an integer $0 < i \ll_r l$ such that $iz \in 4Q + B(0,4)$.  Since $|z| \leq 3R_0$, we have $|iz| \ll_r l R_0$ and thus in fact $iz \in (4Q \cap B(0,O_r(l R_0))) + B(0,4)$.  Thus, to obtain the desired bound on $\N_1(Z)$, it will suffice to show that
$$\N_4(4Q \cap B(0,O_r(l R_0))) \ll_{r} l^{3} k^{-1} \D(Q).$$

Let $Z'$ be any $4$-net of $4Q \cap B(0,O_r(l R_0))$.  Observe that the sets $\zeta' + (Q \cap B(0,1))$ for $\zeta' \in Z'$ are disjoint and lie in $5Q \cap B(0,O_r(l R_0))$.  Thus we have
$$ (\# Z') \#(Q \cap B(0,1)) \leq \#(5Q \cap B(0,O_r(l R_0))).$$

Since $\D(Q) = \frac{\# Q}{\#(Q \cap B(0,1))}$, it suffices to show that

$$\#(5Q \cap B(0,O_r(l R_0))) \ll_{r}  l^{2}k^{-1} \# Q.$$

But (as we are working on the plane) we can cover $B(0,O_r(l R_0))$ by $O_r(l^{2})$ balls of radius $R_0/4$, so by Lemma \ref{pig} we have
$$ \#(10Q \cap B(z_0,R_0/2)) \gg_r l^{-2} \#(5Q \cap B(0,O_r(R_0))).$$
Comparing this with \eqref{qbr-2} we obtain the claim.
\end{proof}

\section{Proof of the inverse theorems}\label{sec10}

We first prove Theorem \ref{ilo}. The proof of Theorem \ref{ilo-sparse} can be obtained with some minor modifications.

Let us  begin with a simple reduction.
Since $\a$ has $O(1)$-controlled second moment, from Chebyshev's inequality we see that $ |\a| \geq n^{A+10} $ with probability $O(n^{-2A-20})$.  Thus if we let $\a'$ be $\a$ conditioned on the event $|\a| \leq n^{A+10}$, we see from the union bound that $p_{\beta,\a}(\v)$ and $p_{\beta,\a'}(\v)$ differ by at most $O( n^{-2A-19} )$.  Thus (modifying $p$ slightly if necessary) we may replace $\a$ by $\a'$, and so we may assume for the rest of the proof that
\begin{equation}\label{polysize}
|\a |  \le  n^{A+10}  = n^{O(1)}  \hbox{ almost surely}.
\end{equation}

Consider a point  $\v$  in $S_{n, \a,\beta,p}$.  Let  $\V =
(V_1,\ldots,V_n)$ be the vector obtained from
 $\beta^{-1} \v/2$ by rounding the coordinates to  the nearest Gaussian integer multiple of $n^{-A-20}$.
 Clearly thus $|\V| = \Theta (\beta^{-1})$. Furthermore, by \eqref{polysize}
$$ p_{1,\a}(\V) \geq p_{\beta,\a}(\v) \geq p.$$

By Lemma \ref{concball}, it follows that
$$ \P_1(\V) \gg p.$$

We are going to find a small $O(1)$-net (in the $l^\infty$ norm) for  the set of all possible $\V$ satisfying the last inequality.
Set $k := n^{1/2 -\eps}$, and let $d \geq 1$ be an integer to be chosen later ($d$ will be bounded by a constant.)

Now we perform the following algorithm (following the proof of \cite[Theorem 2.4]{TVsing}) to construct some elements $w_1,\ldots,w_r$ in $\V$ for some $0 \leq r \leq d$.

\begin{itemize}
\item[Step 0] Initialize $r=0$. Set $\V^{[0]}=\V$.
\item[Step 1] Count how many $V_{j } \in \V^{[r]}$ there are such that
$$\D(\GAP((w_1, \ldots, w_r, V_j),k)) \geq n^\eps \D(\GAP((w_1, \ldots, w_r),k)).$$
If this number is less than $k^2$ then \textbf{STOP}. Otherwise, move on to Step 2.

\item[Step 2] Applying Lemma \ref{pumar-lemma}(v), we can find  some $V_j \in V^{[r]}$ such that
$$\D(\GAP((w_1, \ldots, w_r, V_j),k)) \geq n^\eps \D(\GAP((w_1, \ldots, w_r),k))$$
and
$$ \P_{1}( \V^{[r]} w_1^{k^2} \ldots w_r^{k^2} ) \leq \P_{1}( \V^{[r+1]} w_1^{k^2} \ldots w_r^{k^2} V_j^{k^2} ),$$

\noindent where $\V^{[r+1]}$ is obtained from $\V^{[r]}$ by deleting
a set of $k^{2}$ elements. We then set $w_{r+1} := V_j$ and then
increment $r$ to $r+1$.  If $r=d$ then \textbf{STOP} (with an
error); otherwise return to Step 1.
\end{itemize}

By induction, at each stage in this algorithm we have
$$ \P_{1}( \V^{[r]} w_1^{k^2} \ldots w_r^{k^2} )  \gg p$$
and hence by Theorem \ref{flot} and Lemma \ref{pumar-lemma}(ii)
$$ \D(\GAP((w_1, \ldots, w_r),k)) \ll p^{-1/(1-\eps)} \ll n^{O(\eps)} p^{-1} = n^{O(1)}.$$

On the other hand, by construction we have
$$ \D(\GAP((w_1, \ldots, w_r),k)) \geq n^{r \eps}.$$
Thus, the algorithm must terminate in Step 1 for some $r
=O_{\eps}(1)$.  At this point, we have obtained a tuple $(w_1,
\ldots, w_r)$ of elements in $\V$ with $r = O_\eps(1)$ such that
\begin{equation}\label{wowr}
 \D(\GAP((w_1, \ldots, w_r),k)) \ll_\eps n^{O(\eps)} p^{-1}
\end{equation}
and such that
$$ \D(\GAP((w_1, \ldots, w_r, V_j),k )) < n^\eps \D(\GAP((w_1, \ldots, w_r),k))$$
for all but at most $rk^2 = O_{\eps} (n^{1-2\eps}) \le n^{1-\eps}$ values of $j$.

Now we have enough information to construct the net. First we show
that it costs a relatively small factor to take care of the
exceptional coordinates.
 There are at most $O_{\eps}(k^2) \le n^{1-\eps}$ exceptional values of $j$; we can fix the values of the exceptional $j$ by paying a factor of
 $$ \sum_{i=0}^{n^{1-\eps}} \binom{n}{i} = \exp(o(n)).$$
 For each exceptional $j$, $V_j$ is a Gaussian integer multiple of $O(n^{-A-20})$ of magnitude $O(\beta^{-1})$.
 Thus, the number of possible choices for $V_{j}$ is $\beta^{-1} n^{O(1)}$. So, after we fix the exceptional coordinates $j$, there are at most
 $$(\beta^{-1} n^{O(1)}) ^{n^{1-\eps}} = \exp(o(n)) $$
 \noindent ways to specify the values of these coordinates.

   As for the remaining (non-exceptional)
   coordinates $V_j$, Lemma \ref{lemma:weakz}  (along with \eqref{wowr}, the definition of $k$, and the bound $r=O_\eps(1)$)
   shows that each such $V_j$ lies within distance $O(1)$ of a set of cardinality $ 1+ O_{\eps} (n^{-1/2 + O(\eps)} p^{-1} )$.
  The set of all vectors $V$ has a $O(1)$-net in the $l^\infty$ norm of size
  at most
  $$\exp(o(n)) \left(1+ O_{\eps} (n^{-1/2 + O(\eps)} p^{-1}) \right)^{n} =O( n^{(-1/2+ O(\eps))n} p^{-n} ) + \exp(o(n))$$
  assuming $n$ sufficiently large depending on $p, \eps$.

  \noindent Changing a $O(1)$-net to a $1$-net costs only a $O(1)$ factor. Thus, we can conclude that there is an $1$-net of size at most
$O(n^{(-1/2+ O(\eps))n} p^{-n}) + \exp(o(n))$. As we can choose $\eps$ arbitrarily small, the proof of Theorem \ref{ilo} is complete.

\vskip2mm

To prove Theorem \ref{ilo-sparse}, we just use the sparse version of all lemmas used in the previous proof, except that we take $k$ equal to $\sqrt{m/\mu}$ rather than $n^{1-\eps}$. The starting point is
$$ \P_{\mu} (\V) \gg p.$$

Instead of $\D(\GAP((w_1, \ldots, w_r),k)) $, we will consider  $\D(\GAP((w_1, \ldots, w_r), \sqrt \mu k)) $. Thus, the gain from  Lemma \ref{lemma:weakz} is no longer $k^{-1}$ (which used to lead to the term $n^{-1/2+ O(\eps)}$ in the final bound), but instead $(k \sqrt \mu)^{-1}$ (which leads to the term $n^{O(\eps)} (\sqrt m)^{-1}$). Meanwhile, the $\exp(o(n))$ factor is replaced with $\exp( n^{O(\eps)} k^2) = \exp(n^{O(\eps)} m / \mu)$. The reader is invited to work out the simple details.

\section{Proof of Theorems \ref{lsv} and \ref{lsv-sparse}} \label{lsv-sec} 

Theorem \ref{lsv} follows from the following.
Let $\sigma_{n} (M)$ denote the least singular value of a  matrix $M$ of order $n$.
We shall abbreviate $N = N_{n,\rho}$.

\begin{theorem}\label{lsvprecise} Let $\gamma \geq 0$ be such that
\begin{itemize}
\item  $\|M +N\| \le n^{\gamma}$ with probability one.
\item $|\a_{1}| +\dots + |\a_{n}| \le n^{\gamma}$ with probability one.
\end{itemize}
Then for any $A, B \geq 0$ with
\begin{equation}\label{B-bound}
B > 2\gamma A + 3\gamma + 1/2
\end{equation}
we have
$$\P( \sigma_{n }(M + N) \le n^{-B} ) \ll_{A,B,\gamma,\kappa} n^{-A}.$$
\end{theorem}

Indeed, as $\a$ has finite second moment, we can assume that $|\a| \le n^{A+10}$, at the cost of a (negligible) additional term $o(n^{-A})$ in probability. Thus, by restricting $\a$ to the event $|\a| \le n^{A+10}$ and using the assumption about $M$ in Theorem \ref{lsv}, we can satisfy both assumptions in Theorem \ref{lsvprecise}, for $\gamma$ large enough.

\begin{remark}
We can have a more efficient form of the theorem by bounding the
probability that the two assumptions on $\| M +N\|$ and
$|\a_{1}|+\dots + |\a_{n}|$ fail (rather than assuming that they
hold with probability one). The relation between $B$ and $\gamma, A$
can be strengthened and we will do that in another paper.
\end{remark}

We now prove Theorem \ref{lsvprecise}.  We suppress all dependence of the implied constants
on $A,\gamma,B,\kappa$.

Let us call a unit vector $\v = (v_1,\ldots,v_n)$ \emph{poor} if we
have
$$ p_{n^{-B+1/2},\a}(\v) \leq n^{-A-1},$$
and \emph{rich} otherwise. Theorem \ref{lsvprecise} follows directly from the following
two lemmas and the fact that
$$\sigma_{n } (M) = \inf _{|\v|=1} | M\v| . $$

\begin{lemma}[Poor vectors are negligible]  \label{lemma:poor} We have
$$ \P( \| (M+N) \v \| \leq n^{-B} \hbox{ for some poor unit vector } \v ) \ll n^{-A}.$$
\end{lemma}

\begin{lemma}[Rich vectors are negligible]  \label{lemma:rich} We have
$$ \P( \| (M+N) \v \| \leq n^{-B} \hbox{ for some rich unit vector } \v ) \ll n^{-A}.$$
\end{lemma}

\begin{proof}  (Proof of Lemma \ref{lemma:poor})
We repeat the argument from \cite{TVstoc}. Let $E$ be the event that
$\|(M+N) v \| \leq n^{-B}$ for some poor unit vector $\v$. If $E$
holds, then the least singular value of $M+N$ is at most $n^{-B}$,
and so the same is true for the adjoint $(M+N)^\dagger$.  Thus there
exists a row vector $\w^\dagger$ such that
\begin{equation}\label{wdag}
 \| \w^\dagger (M+N) \| \leq n^{-B}.
\end{equation}
Write $\w^\dagger = (w_1,\ldots,w_n)$.  By paying a factor of $n$
and using symmetry we may assume that the last coefficient of
$\w^\dagger$ has the largest magnitude, thus
\begin{equation}\label{wbig}
|w_n| \geq |w_i| \hbox{ for all } i.
\end{equation}
In particular, we have
\begin{equation}\label{wnbig}
|w_n| \geq 1/\sqrt{n}.
\end{equation}
Thus, if we let $F$ be the event that there exists a unit vector
$\w$ obeying both \eqref{wdag} and \eqref{wbig}, we have
\begin{equation}\label{pef}
\P( E) \leq n \P( E \wedge F ).
\end{equation}
Let $X_1,\ldots,X_n$ be the rows of $M+N$.  We shall condition on the first $n-1$ rows $X_1,\ldots,X_{n-1}$.  Observe that if $E$ holds,
then there exists a poor unit vector $\v$ such that
$$ (\sum_{i=1}^n |X_i \cdot \v|^2)^{1/2} = \| (M+N) \v \| \leq n^{-B}.$$
Thus, if $\P(E|X_1,\ldots,X_{n-1})$ is non-zero, then there exists a
poor unit vector $\u$ such that
\begin{equation}\label{xu}
 (\sum_{i=1}^{n-1} |X_i \cdot \u|^2)^{1/2} \leq n^{-B}.
\end{equation}
On the other hand, if $F$ holds, and $\w^\dagger = (w_1,\ldots,w_n)$
is as above, then by \eqref{wdag}
$$ \| w_1 X_1 + \ldots + w_n X_n \| \leq n^{-B};$$
taking inner products with the unit vector $\u$ and using the
triangle inequality, we conclude
$$ |w_n| |X_n \cdot \u| \leq \sum_{i=1}^{n-1} |w_i| |X_i \cdot \u| + n^{-B}.$$
Using \eqref{wnbig}, Cauchy-Schwarz, and \eqref{xu} we conclude
$$ |X_n \cdot \u| \leq n^{-B+1/2} + n^{-B+1/2} = 2 n^{-B+1/2}.$$
On the other hand, since $\u$ is poor, and $X_n$ is independent of
$X_1,\ldots,X_{n-1}$ (and hence independent of $\u$ also), we have
$$ \P( |X_n \cdot \u| \leq 2 n^{-B+1/2} | X_1,\ldots,X_{n-1} ) \leq n^{-A-1}.$$
Putting all this together, we conclude that
$$ \P( E \wedge F | X_1,\ldots,X_{n-1} ) \leq n^{-A-1}$$
uniformly in the choice of $X_1,\ldots,X_{n-1}$. Integrating over
$X_1,\ldots,X_{n-1}$ and using \eqref{pef} we obtain
$\P(E) \leq n^{-A}$, as desired.
\end{proof}

\begin{proof} (Proof of Lemma \ref{lemma:rich})
Let $\eps > 0$ be a sufficiently small constant (in particular, smaller than the constant in Theorem \ref{ilo}); we allow all implied constants to depend on $\eps$.  We may also assume that $n$ is sufficiently large depending on $\eps$.

Let $J$ be the smallest integer strictly larger than
$2A+2$, thus $2A+2 < J \leq 2A+3$. Thus, if we set $\delta := (A+1)/J$, then (using \eqref{B-bound}) we have $0 < \delta < 1/2$ and $B > J \gamma$.  If $\eps$ is sufficiently small, we thus have
\begin{equation} \label{smalldelta}
 B > JC + 1/2 \hbox{ and } \delta + 3\eps  < 1/2
\end{equation}
where
\begin{equation}\label{C-def}
C  := \gamma + 2 \eps.
\end{equation}

Let $\v$ be a rich unit vector.  For $j=0,1,\ldots, J$, consider the
quantities
$$ p_{n^{-B+Cj+1/2},\a}(\v).$$

These quantities are increasing in $j$, and range between $n^{-A-1}$
and $1$ since $\v$ is rich.  Applying the pigeonhole principle and
using the definition of $\delta$, we can thus find a positive $0 \le
j \le J-1$  such that
$$  p_{n^{-B+C(j+1)+1/2},\a}(\v) \leq n^{\delta} p_{n^{-B+Cj+1/2},\a}(\v).$$

Define, for any $0\le j \le J-1$ and $1\le k \le \lceil (A+1)/\eps \rceil$, the set $\Omega_{j,k} $ as
$$ \Omega_{j,k} := \{\v| (\v \,\, \hbox{\rm rich})\,\, \wedge
(p_{n^{-B+C(j+1)+1/2},\a}(\v) \leq n^{\delta}
p_{n^{-B+Cj+1/2},\a}(\v)) \wedge (p_{n^{-B+Cj+1/2},\a}(\v) \in
[n^{-k\eps}, n^{-(k-1)\eps})  ) \} .$$

Since the number of pairs $j,k$ is $O(1)$, it suffices by the union bound to show that for each fixed $j,k$
\begin{equation}\label{puma}
 \P( \| (M+N) \v \| \leq n^{-B} \hbox{ for some  unit vector } \v \in \Omega_{j,k} ) =o( n^{-A}).
 \end{equation}

In fact, we are going to show that this probability is exponentially small.

Let $p:= n^{-k\eps}$.  In the notation of Theorem \ref{ilo}, $\v$
lies in $S_{n,\a,n^{-B+Cj+1/2},p}$. Thus by this theorem, there is a
set $V$ of cardinality at most
$$ \# V \ll  n^{-n/2 +  \eps n} p^{-n} + \exp(o(n))$$
\noindent such that for each $\v \in \Omega_{j,k}$ there is $\v' \in
V$ such that $\| \v-\v'\| _{\infty} \le n^{-B+Cj+1/2}$.

Consider $v \in \Omega_{{j,k}}$ and $\v' \in V$ as above. Recall
that $\| M+N \| \leq n^\gamma$ almost surely. Thus with probability
$1$ we have
$$ \| (M+N) (\v-\v') \| \leq n^{-B+Cj + 1 + \gamma}.$$

By the triangle inequality, we have
$$ \| (M+N) \v' \| \leq 2n^{-B+Cj + 1 +\gamma} .$$

As usual, let $X_{i}$ be the $i$th row of $M+N$.
It follows that  there are at least $n':= n - n^{1-\eps}$ coordinates $1\le i \le n$ such that
$$ |X_{i} \cdot \v' |  \le n^{-B+ Cj + 1/2 + \gamma+ \eps}. $$

Now we relate the probability that $ |X \cdot \v' |  \le n^{-B+ Cj +
1/2 + \gamma+ \eps}$ with $p$, where $X:= (\a_{1}, \dots, \a_{n})$.
Consider the quantity

$$|X\cdot \v - X \cdot \v' | = |(v_{1}-v'_{1}) \a_{1 } +  \cdots \cdot (v_{n} -v'_{n})\a_{n}|.$$
\noindent Notice that  $|v_{i}-v_{i}'| \le n^{-B+Cj+1/2}$ and also
$\sum_{i} |\a_{i}| \le n^{\gamma}$ with probability one. Thus
$$|X\cdot \v - X \cdot \v' |  \le n^{-B+Cj+1/2+\gamma} $$
\noindent which implies, through the triangle inequality, that
$$|X \cdot \v|  \le n^{-B+Cj+1/2 +\gamma}+ n^{-B+ Cj+1/2  + \gamma+ \eps} \le n^{-B + C(j+1)+1/2},$$

\noindent where in the last inequality we used \eqref{C-def}.

We can then conclude that
$$ \P(|X_{i} \cdot \v' |  \le n^{-B+ Cj + 1/2 + \gamma+ \eps}) \leq p_{n^{-B+C(j+1)+1/2}, \a}(\v)
\leq n^{\delta} p_{n^{-B+Cj+1/2},\a}(\v) \leq  n^{\delta+\eps} p,$$

\noindent where in the last inequality we used the definition of $\Omega_{j,k}$.

Also, a very crude second moment argument, using the fact that $\a$ has $\kappa$-controlled
second moment, gives
\begin{equation}\label{crude}
p_{n^{-B+C(j+1)+1/2},\a}(\v) \leq 1-\delta'
\end{equation}
if $\delta' > 0$ is small enough depending on $\kappa$.  Thus
$$\P(|X_{i} \cdot \v' |  \le n^{-B+ Cj + 1/2 + \gamma+ \eps})  \leq
\min(  n^{\delta+\eps} p, 1-\delta' ).$$

By the union bound, we thus have
$$ \P\left( \| (M+N) \v' \| \leq n^{-B+Cj +1 + \gamma + \eps} \right) \leq
\min( n^{\delta+\eps} p, 1-\delta')^{n'} \binom{n}{n'}. $$
Again by the union bound, the left-hand side of \eqref{puma} is at most
\begin{eqnarray*} & & \left(n^{-n/2 + \eps n } p^{-n} + \exp(o(n))\right)
\min\left( n^{\delta+ \eps} p,1-\delta'\right)^{n'} \binom{n}{n'}  \\
& \ll& n^{-n/2 + \eps n } p^{-n} (n^{\delta+ \eps} p)^{n'} \binom{n}{n'}
+ \exp(o(n)) (1-\delta')^{n'} \binom{n}{n'}. \end{eqnarray*}

It is routine to verify that the last  quantity  is $o(n^{-A})$
(indeed, we obtain a bound of the form $O(\exp( - \sigma n))$ for some $\sigma > 0$).
Our proof is complete.
 \end{proof}

\subsection{The sparse case}

Now we sketch the proof of Theorem \ref{lsv-sparse}.  We repeat the
above arguments with the following changes.  We will of course
replace Theorem \ref{ilo} by Theorem \ref{ilo-sparse}, with $\mu :=
\rho$. Due to the presence of the additional factors of $\mu$ in
that theorem, we can no longer afford to choose $\delta$ close to
$1/2$, so we instead choose $\delta$ to be very small, say $\delta =
\eps$, where $\eps$ is very small compared to $1-\alpha$. In order
to take $\delta$ this small, we will need $B$ to be much larger than
what \eqref{B-bound} requires, but this is not a problem. For our
applications, all we need is that $B$ does not depend on $n$.

The treatment of the poor vectors (Lemma \ref{lemma:poor}) in the
sparse case is the same as in the non-sparse case.  The treatment of
the rich vectors (Lemma \ref{lemma:rich}) is also essentially the
same, except for the fact that we no longer have \eqref{crude}. To
be more precise, $1-\delta'$ needs to be replaced by $1-\delta'
\rho$.  In the cases when $k$ is larger than some absolute constant
(say 5), \eqref{crude} is not needed, since in this case $p$ is
sufficiently small and

$$\min (n^{\delta +\eps} p, 1 -\delta' \rho) =\min (n^{2\eps} p, 1 -\delta' \rho)
= n^{2\eps} p $$

and the above argument goes through without difficulty, so long as
one applies Theorem \ref{ilo-sparse} with $m := n^{C_0\eps}$ for
some sufficiently large absolute constant $C_0$.

In the remaining case where $k$ is at most 5, the replacement of
$1-\delta'$ by  $1-\delta' \rho$ becomes too expensive and we will
avoid it by a rescaling argument, using the pigeonhole principle.

 To start, notice that from the
definition of $\Omega_{j,k}$ and the fact $k \le 5$, we have
\begin{equation}\label{pnv}
p_{n^{-B'},\a \I_\rho}(\v) \ge n^{-5 \eps}
\end{equation}
for some fixed $B - O_\eps(1) \leq B' \leq B$.  Since the left-hand
side is $p_{1,\a \I_\rho}(n^{B'} \v)$, we also see from Lemma
\ref{concball-sparse} that

$$ \P_\rho( n^{B'} \v ) \gg n^{-4\eps}.$$

We observe that this implies that $\v$ is ``compressed'' in the
sense that at most $n^{100\eps}/\rho$ of the coefficients of $v =
(v_1,\ldots,v_n)$ can exceed $n^{-B'+10}$ in magnitude. (Of course,
instead of $100$, on can use any large constant.) Indeed, if instead
we had at least $n^{100\eps}/\rho$ coefficients $v_i$ of magnitude
at least $n^{-B'+10}$ for some large absolute constant $A$, we see
from Lemma \ref{pumar-lemma} that

$$ \P_\rho( (n^{B'} v_i)^{n^{A\eps}/\rho} ) \gg n^{-4\eps}$$
for one of these $v_i$, but one can show that this is not the case
by a direct computation using the $\kappa$-controlled second moment
hypothesis, or else by an appeal to Theorem \ref{flot}. (Here we
used the
 notation of Lemma \ref{pumar-lemma}: $(z)^s$ denotes a vector of
 length $s$ whose every coordinate equals $z$.)

We have just seen that
$$ |\{ 1 \leq i \leq n: |v_i| \geq n^{-B'+10} \}| \le n^{100\eps}/\rho.$$

Next, we apply the pigeonhole principle to conclude the existence of a $B''$ with
$B' - O_\eps(1) \leq B'' \leq B'-10$ and an integer $m = O_{\eps,\gamma}(1)$ with
\begin{equation}\label{nmeps}
n^{m\eps} \le n^{100\eps}/\rho
\end{equation}
such that
$$ n^{(m-1)\eps } \leq |\{ 1 \leq i \leq n: |v_i| \geq n^{-B''+10+\gamma} \}| \leq
|\{ 1 \leq i \leq n: |v_i| \geq n^{-B''} \}| \leq n^{m\eps}.$$


By paying a factor of $O_{\eps,\gamma}(1)$ in our final probability
bound we may fix $B''$ and $m$.  If we define the vector $\w$ by
setting $w_i$ to be the nearest (Gaussian integer) multiple of
$n^{-B''+1}$ to $v_i$, we see that $w_i$ is non-zero for at most
$n^{m\eps}$ coordinates $i$, and has magnitude $\gg
n^{-B''+10+\gamma}$ for at least $n^{(m-1)\eps}$ of these
coordinates.  Also, if $\|(M+N) \v\| \leq n^{-B}$, we see from the
triangle inequality and crude computations that $\|(M+N) \w \| \leq
n^{-B'' + 5 + \gamma}$ (say), recalling that $B'' < B+10$.

On the other hand, note that if we let $\I_{i,\rho}$ be independent
samples of $\I_\rho$, then with probability $\Omega (n^{(m-1)\eps}
\rho)$, there is at least one $i$ with $\I_{i,\rho}=1$ and $|w_i|
=\Omega(n^{-B''+10+\gamma})$.  From this we conclude that
$$ p_{n^{-B''+7+\gamma},\a \I_\rho} \geq 1 - \delta' n^{(m-1)\eps} \rho$$
for some absolute constant $\delta' > 0$ (cf. \eqref{crude}), and
thus for each fixed $\w$ we have
$$ \P( \|(M+N) \w \| \leq n^{-B'' + 5 + \gamma} ) \ll \exp( - \Omega( n^{(m-1)\eps} \rho n ) ).$$
On the other hand, a direct counting argument shows that the number
of possible $\w$ is at most $\exp( O( n^{(m+1)\eps} ) )$. Recall
that $\rho \ge n^{-1+\alpha}$ and $\alpha$ is much larger than
$\eps$. It follows that

$$n^{(m-1)\eps} \rho n \gg n^{(m+1)\eps} $$

for any $m$.  Applying the union bound we obtain a suitably small
contribution to the sparse analogue of Lemma \ref{lemma:rich}, as
required.

\section{Proof of the circular law}\label{circular-sec} 

We now use Theorem \ref{lsv} to derive Theorem \ref{circular}.

By Lemma \ref{rot} and rotating $\a$ by a constant phase if necessary\footnote{Here of course we use the obvious fact that the circular law is invariant under phase rotation of the underlying random variable $\a$.}, we may assume that $\a$ has $\kappa$-controlled second moment for some $\kappa$.  Allowing implied constants to depend on this $\kappa$, we thus have that $\a$ has $O(1)$-controlled second moment, which will allow us to apply Theorem \ref{lsv} later.

We closely follow the (now standard) arguments\footnote{One could also follow the approach of G\"otze and Tikhomirov \cite{gotze}, as was done in \cite{PZ}.} in \cite[Chapter 10]{bai} (which are in turn based on the earlier work of Girko \cite{girko} and Bai \cite{bai-first}), which we briefly review here.

Let $c_n: \R \times \R \to \C$ be the characteristic function
$$ c_n(u,v) := \int_\R \int_\R e^{\sqrt{-1} ux + \sqrt{-1} vy} \mu_n(dx, dy)$$
of the ESD $\mu$, and similarly define
$$ c(u,v) := \int_\R \int_\R e^{\sqrt{-1} ux + \sqrt{-1} vy} \mu_\infty(dx, dy)$$
of the uniform measure $\mu_\infty$ on the disk.
The sequence of empirical measures $\mu_n$ can be shown to be a.s. tight just from the assumption that $\a$ has finite second moment (see \cite[Lemma 10.5]{bai} and \cite[Theorem 3.6]{bai}, and also the discussion in \cite[p. 295]{bai}), and so by standard arguments it suffices to show, for almost every $u, v$, that
\begin{equation}\label{cnu}
c_n(u,v) \to c(u,v)
\end{equation}
almost surely.

Henceforth we fix $u, v$.  We can take $uv \neq 0$ since we only need the claim for almost every $u,v$.  From \cite[Lemma 10.2]{bai}, we have the Stieltjes transform identity (first observed by Girko \cite{girko})
\begin{align*}
c_n(u,v) &=
- \frac{u^2 + v^2}{2\sqrt{-1} u \pi} \int_\R \int_\R \frac{1}{n} \sum_{k=1}^n \frac{\Re(\lambda_k - z)}{|\lambda_k - z|^2} e^{\sqrt{-1}us + \sqrt{-1}vt}\ ds dt \\
&=\frac{u^2 + v^2}{2\sqrt{-1} u \pi} \int_\R \int_\R \frac{d}{ds} \log |\det(\frac{1}{\sqrt{n}} N_n - z I_n)| e^{\sqrt{-1}us + \sqrt{-1}vt}\ ds dt \\
&= \frac{u^2 + v^2}{2\sqrt{-1} u \pi} \int_\R \int_\R g_n(s,t) e^{\sqrt{-1}us + \sqrt{-1}vt}\ ds dt
\end{align*}
where $z := s+\sqrt{-1} t$,
$$ g_n(s,t) := \frac{d}{ds}\left(\int_0^\infty \log x\ \nu_n(dx, z)\right)$$
and $\nu_n$ is the empirical distribution of the positive-definite Hermitian matrix
$$ H_n := (\frac{1}{\sqrt{n}} N_n - z I_n) (\frac{1}{\sqrt{n}} N_n - z I_n)^*.$$
The expression $g_n(s,t)$ is absolutely integrable in $s,t$, however because of the unboundedness of $\log x$, Fubini's theorem is not currently applicable, and one must take some care with interchanging integrals or derivatives in this expression. In \cite[Lemma 10.4]{bai}, the analogous identity
$$ c(u,v) = \frac{u^2 + v^2}{2\sqrt{-1} u \pi} \int_\R \int_\R g(s,t) e^{\sqrt{-1}us + \sqrt{-1}vt}\ ds dt$$
is derived, where $g$ is a function whose explicit form (given in \cite[p. 296]{bai}) we will not review here.  The task is then to show that
$$ \int_\R \int_\R (g_n(s,t) - g(s,t)) e^{\sqrt{-1}us + \sqrt{-1}vt}\ ds dt \to 0$$
a.s.

The next steps in \cite{bai} are to perform some truncations in the region of integration.
Let $S > 2$ be any integer.  In \cite[Lemma 10.6]{bai} (see also the discussion in \cite[p. 299]{bai}), it was shown that
$$ \lim_{S \to \infty} \limsup_{n \to \infty}
\left|\int\int_{|s| \geq S \hbox{ or } |t| \geq S^3} (g_n(s,t) - g(s,t)) e^{\sqrt{-1}us + \sqrt{-1}vt}\ ds dt\right| = 0$$
a.s.  Thus it suffices to show that
$$ \int\int_{|s| \leq S, |t| \leq S^3} (g_n(s,t) - g(s,t)) e^{\sqrt{-1}us + \sqrt{-1}vt}\ ds dt \to 0$$
a.s. for every $S > 2$.

Fix $S$.  For any $\eps > 0$, let $T \subset \R^2$ denote the set
$$ T := \{ (s,t) \in \R^2| |s| \leq S, |t| \leq S, ||z|-1| \geq \eps \}$$
(recall that $z := s+ \sqrt{-1}t$).  In \cite[Lemma 10.7]{bai} it is
shown that
$$ \int\int_{|s| \leq S, |t| \leq S^3: (s,t) \not \in T} |g_n(s,t)|\ ds dt \leq 32 \sqrt{\eps}$$
and similarly with $g_n(s,t)$ replaced by $g(s,t)$; thus it suffices to show that
$$ \int\int_T (g_n(s,t) - g(s,t)) e^{\sqrt{-1}us + \sqrt{-1}vt}\ ds dt \to 0$$
a.s. for each $\eps > 0$.

Fix $\eps > 0$.  Recall that
$$ g_n(s,t) := \frac{d}{ds}\left(\int_0^\infty \log x\ \nu_n(dx, z)\right).$$
In \cite[Lemma 10.10]{bai} it is shown that
$$ g(s,t) := \frac{d}{ds}\left(\int_0^\infty \log x\ \nu(dx, z)\right)$$
where $\nu$ is an explicit probability measure which we will not review here; in particular, the inner integral is absolutely convergent.  Set $\eps_n := n^{-2B}$ for some large absolute constant $B$ (independent of $n$) to be chosen later. Using the integration by parts argument given in \cite[\S 10.7]{bai}, it suffices to show that
\begin{equation}\label{muso}
 \limsup_{n \to \infty} \int\int_T \left|\int_{\eps_n}^\infty \log x\ (\nu_n(dx, z) - \nu(dx, z))\right| = 0
\end{equation}
and
\begin{equation}\label{muso-2}
 \limsup_{n \to \infty} \int\int_T \left|\int_0^{\eps_n} \log x\ \nu_n(dx, z)\right| = 0
\end{equation}
and similarly with the two-dimensional integral on $T$ replaced by one-dimensional integrals on the boundary of $T$.  We shall only estimate the two-dimensional integrals, as the treatment of the one-dimensional ones are similar\footnote{Actually, by employing a smooth cutoff to $T$ rather than a rough one, one can dispense with the need to consider boundary integrals.}.

We first prove \eqref{muso}.  Since $\a$ has finite second moment, a simple application of Chebyshev's inequality and the Borel-Cantelli lemma, and crude bounds on the spectral norm of $N_n$ shows that almost surely $\nu_n$ is supported on the interval $[0, n^{100}]$.  Thus it suffices to show that
$$ \limsup_{n \to \infty} \int\int_T |\int_{\eps_n}^{ n^{100}} \log x\ (\nu_n(dx, z) - \nu(dx, z))| = 0$$
Observe that $\log x$ has total variation bounded by a finite multiple of $\log n$ on the $x$ region of integration, thus it will suffice to show that
\begin{equation}\label{moosh}
\limsup_{n \to \infty} (\log n) \sup_{z \in T} \sup_{x > 0} |\nu_n(x,z) - \nu(x,z)| = 0.
\end{equation}
For this, it is convenient to perform some truncation, following \cite[\S 10.5.1]{bai}.  Let $0 < \delta < 1/4$ be arbitrary, and define the truncated random variables $\hat a_{ij}$ (depending on $n$) by
$$ \hat a_{ij} := a_{ij} \I( |a_{ij}| < n^\delta ) - \E\left( a_{ij} \I( |a_{ij}| < n^\delta )\right)$$
and the normalised random variable $\tilde a_{ij}$ by
$$ \tilde a_{ij} := \hat a_{ij} / \sqrt{\E(|\hat a_{ij}|^2)}.$$
One easily verifies that $\tilde a_{ij}$ has mean zero, variance $1$, and is also bounded by $O(n^\delta)$ almost surely.  Let $\tilde N_n$ be the matrix with entries $\tilde a_{ij}$, let $\tilde H_n$ be the positive-definite matrix
$$ \tilde H_n := (\frac{1}{\sqrt{n}} \tilde N_n - z I_n) (\frac{1}{\sqrt{n}} \tilde N_n - z I_n)^*$$
and let $\tilde\nu_n(x,z)$ be the distribution function associated to $\tilde H_n$.  The argument in \cite[\S 10.5.1]{bai} gives the bound
\begin{equation}\label{ncde}
 \limsup_{n \to \infty} n^{c \delta \eta} \sup_{z \in T} L( \nu_n(\cdot,z), \tilde \nu_n(\cdot,z) ) < \infty
\end{equation}
almost surely for some absolute constant $c > 0$, where $L$ denotes the Levi distance.  Next, from \cite[Lemma 10.15]{bai} we have
\begin{equation}\label{10.15}
\limsup_{n \to \infty} n^{c' \delta \eta} \sup_{z \in T} \sup_{x>0} |\tilde \nu_n(x,z) - \nu(x,z)| < \infty
\end{equation}
almost surely for another absolute constant $c' > 0$ (this is where we use the hypothesis $\delta < 1/4$).  Applying \cite[Lemma 12.18]{bai} we conclude
$$ \limsup_{n \to \infty} n^{c' \delta \eta} \sup_{z \in T} L(\tilde \nu_n(\cdot,z), \nu(\cdot,z)) < \infty$$
and hence by the triangle inequality for Levi distance
$$ \limsup_{n \to \infty} n^{c'' \delta \eta} \sup_{z \in T} L(\nu_n(\cdot,z), \nu(\cdot,z)) < \infty$$
for some $c'' > 0$.
Applying \cite[Lemma 12.18]{bai} and \cite[Lemma 10.8]{bai} we obtain
$$ \limsup_{n \to \infty} n^{c''' \delta \eta} \sup_{z \in T} \sup_{x > 0} |\nu_n(x,z) - \nu(x,z)| < \infty$$
for some $c''' > 0$, which yields \eqref{moosh} (with some room to spare).  This proves \eqref{muso}.

The only remaining task is to prove \eqref{muso-2}.  We would like to reduce matters to establishing that almost surely we have
\begin{equation}\label{nux}
 \lim_{n \to \infty} \int_0^{\eps_n} \log x\ \nu_n(dx, z) = 0
 \end{equation}
for almost every $z$.  The Lebesgue dominated convergence theorem does not apply directly. However, observe from the triangle inequality in $L^2$ that
\begin{align*}
\left(\int\int_T \left|\int_0^\infty \log x\ \nu_n(dx, z)\right|^2\right)^{1/2}
&= 2 \left(\int\int_T \left|\frac{1}{n} \sum_{k=1}^n \log |\lambda_k - z|\right|^2\right)^{1/2} \\
&\leq 2 \frac{1}{n} \sum_{k=1}^n (\int\int_T \left|\log |\lambda_k - z|\right|^2)^{1/2} \\
&\ll_T 1
\end{align*}
since $\log |z|$ is locally square-integrable.  From bounds on $\nu$ (e.g. \cite[Lemma 10.8]{bai} and the estimates used to prove \cite[Lemma 10.10]{bai}), we also have
$$ \left(\int\int_T |\int_{\eps_n}^\infty \log x\ \nu(dx, z)|^2\right)^{1/2} < \infty,$$
which by \eqref{moosh} implies that
$$ \left(\int\int_T |\int_{\eps_n}^\infty \log x\ \nu_n(dx, z)|^2\right)^{1/2} $$
is bounded uniformly in $n$.  Thus
\begin{equation}\label{intt}
 \left(\int\int_T |\int_0^{\eps_n} \log x\ \nu_n(dx, z)|^2\right)^{1/2}
\end{equation}
is bounded uniformly in $n$, which implies that the sequence of functions $\int_0^{\eps_n} \log x\ \nu_n(dx, z)$ is uniformly integrable on $T$.

Now we can deduce \eqref{muso-2} from \eqref{nux}.  To see this, let $M > 1$ be a large parameter, and let $T_{M,n}$ be the set of all $z$ such that $|\int_0^{\eps_n} \log x\ \nu_n(dx, z)| \leq M$. From \eqref{nux} and the Lebesgue dominated convergence theorem we have
$$ \lim_{n \to \infty} \int\int_{T_{M,n}} \left|\int_0^{\eps_n} \log x\ \nu_n(dx, z)\right| = 0.$$
On the other hand, from the uniform boundedness of \eqref{intt} we see that
$$ \limsup_{n \to \infty} \int\int_{T \backslash T_{M,n}} \left|\int_0^{\eps_n} \log x\ \nu_n(dx, z)\right| \ll \frac{1}{M}.$$
Adding these two estimates, and then letting $M \to \infty$, we obtain \eqref{muso-2}.

It remains to prove \eqref{nux}.  By Fubini's theorem, it suffices to show for every $z$ that \eqref{nux} holds almost surely.  But observe that the integrand in \eqref{nux} vanishes whenever $\frac{1}{\sqrt{n}} N_n - zI_n$ has least singular value at least $n^{-B}$.  By Theorem \ref{lsv}, this holds with probability at least $1 - O(n^{-100})$, if $B$ is sufficiently large.  The claim then follows from the Borel-Cantelli lemma.  This concludes the proof of Theorem \ref{circular}.

\section{Relaxation of the moment condition}\label{relax-sec} 

We observe  that the bound \eqref{moosh} was established with some room to spare.  In fact, the arguments in \cite{bai} allow one to relax the condition $\E |\a|^{2+\delta} < \infty$ to the slightly weaker condition

\begin{equation}\label{alog}
\E |\a|^2 \log^C( 2 + |\a|) < \infty
\end{equation}

\noindent for any sufficiently large constant $C$. By inspecting the arguments in \cite{bai},  we see that any $C > 16$ will work. Perhaps a better constant can be obtained by
tightening some calculations, but we do not try to pursue this direction. It seems to us that in the current approach, the extra log term cannot be removed completely in order to
establish the full conjecture.

We sketch the necessary changes to the argument as follows.  The only part of the argument which needs any attention at all is the proof of \eqref{moosh} (the remaining components of the argument work even just assuming finite second moment for $\a$).  We fix $\delta$ close to $1/4$.  We argue as before but with $\eta$ set equal to $\eta := \frac{C \log \log n}{\log n}$, thus $\eta$ now decays slowly in $n$.  One easily checks using \eqref{alog} that $\E |\tilde \a|^{2+\eta}$ is bounded uniformly in $n$ by some bound $B$.  We then use the arguments from \cite{bai} as before, noting that $n^{c''' \delta \eta}$ will grow faster than $\log n$ if $C$ is chosen large enough.  Almost all of the arguments in \cite{bai} go through even when $\eta$ depends on $n$.  The one task which requires some care is the verification of \eqref{ncde}.  Following the arguments in \cite[\S 10.5.2]{bai} to prove \eqref{ncde}, everything goes through without difficulty except one step in the proof of \cite[Lemma 10.13]{bai}, in which one needs to establish that
$$ \sum_{n=1}^\infty \P\left( \sup_{z \in T} \| \tilde H_n \| > \frac{1}{2} n^{2\delta \eta} \right) < \infty$$
(in order to use the Borel-Cantelli lemma to neglect the contribution of $\I(  \| \tilde H_n \| > \frac{1}{2} n^{2\delta n} )$ for all $z \in T$; note that the computation in \cite[p. 311]{bai} here contains some minor typographical errors).  Writing $\tilde H_n$ in terms of $\tilde N_n$ and using the triangle inequality to dispose of the $zI_n$ terms, it suffices to show that
$$ \sum_{n=1}^\infty \P\left( \| \tilde N_n \| > \frac{1}{4} n^{\delta \eta} \sqrt{n} \right) < \infty$$
Using the moment method, we obtain the bound
$$ \P\left( \| \tilde N_n \| > \frac{1}{4} n^{\delta n}\right) \leq \left(\frac{1}{4} n^{\delta \eta}\right)^{-2k} n^{-k} \E \tr( (\tilde N_n \tilde N_n^*)^k )$$
for any integer $k \geq 1$.
If we choose $k := \lfloor \frac{K \log n}{\log \log n} \rfloor$ for some sufficiently large absolute constant $K$, then the factor $(\frac{1}{4} n^{\delta \eta})^{-2k}$ becomes $O( n^{-100} )$.

To conclude the argument, it suffices to show that
\begin{equation}\label{trak}
 \E \tr( (\tilde N_n \tilde N_n^*)^k ) \ll_{B,K} O(n)^{k+1}.
\end{equation}

This type of bound was established for bounded $k$ in \cite[Lemma 10.11]{bai} using the moment method.
But it is well-known that the method extends to much higher value of $k$, in particular  $k = O_K( \frac{\log n}{\log\log n})$.
Indeed, the left-hand side of \eqref{trak} can be expanded as

\begin{equation}\label{eaa}
\sum_{i_1,\ldots,i_k,j_1,\ldots,j_k \in \{1,\ldots,n\}} \E \tilde \a_{i_1 j_1} \overline{\tilde \a_{j_1 i_2}} \ldots \tilde \a_{i_k j_k} \overline{\tilde \a_{j_k i_1}}.
\end{equation}

To estimate  this, we consider the closed walk of length $2k$ on the set $\{1,\ldots,n\} \times \{1,2\}$, in which one walks from $(i_1,1)$ to $(j_1,2)$ to $(i_2,1)$ to $(j_2,2)$, and so forth to $(i_k,1)$ to $(j_k,2)$ and then back to $(i_1,1)$.  If there is any edge traversed exactly once then the summand in \eqref{eaa} vanishes (since $\tilde \a$ has mean zero, and since all the $\tilde \a_{ij}$ are independent).  Thus we may assume each edge is traversed at least twice, thus there are at most $k$ edges traversed, and thus at most $k+1$ vertices.  Suppose in fact that there are $l+1$ vertices traversed for some $1 \leq l \leq k$.  Then there are at least $l$ edges traversed, and so the sum over all edges of the multiplicity minus two is at most $2k-2l$.  Since $\tilde a$ has a second moment of $O(1)$ and is bounded by $O(n^\delta)$, we conclude that the summand in \eqref{eaa} is at most $O(1)^k O( n^{\delta} )^{2k-2l}$.  On the other hand, the number of closed walks of length $2k$ in a set of $2n$ vertices which traverse exactly $l+1$ vertices can be computed to be at most
  $$(2n)^{l+1} 4^l \binom{2k}{2l} (l+1)^{4(k-l)}$$
(see \cite{furedi} or the introduction of \cite{Vu1}).

Thus the total contribution to \eqref{eaa} can be bounded by
$$ \sum_{l=1}^k O(1)^k O( n^{\delta} )^{2k-2l} (2n)^{l+1} 4^l \binom{2k}{2l} (l+1)^{4(k-l)}.$$
The last ($l=k$) term (which is the dominating term)  is of order  $O(1)^k n^{k+1}$, which is acceptable.  As for the $l < k$ terms, we can bound their contribution crudely by

$$ \sum_{l=1}^{k-1} O(1)^{k} O( k n^{\delta} )^{2k-2l} n^{l+1} =o( n^{k+1}),$$

\noindent using the definition of $k$ and the fact that $\delta$ is small. Thus, this contribution is negligible compared to the main term.  This proves \eqref{trak}, and completes the derivation of the circular law under the hypothesis \eqref{alog}.

\section{Rate of Convergence}\label{sec14}

Let us return to the original hypothesis of bounded $(2+\eta)^\th$ moment for some fixed $\eta > 0$.
The above arguments can be pursued in more detail to obtain the more quantitative result that with probability $1$, we have
\begin{equation}\label{supmun}
\sup_{s,t} |\mu_n(s,t) - \mu_\infty(s,t)| \ll n^{-\eta'}
\end{equation}
for some $\eta' > 0$ depending on $\eta$, and all sufficiently large $n$.

A full exposition of this improvement would be very tedious, so we only give a brief sketch of how the argument proceeds.  We first make some Fourier-analytic reductions, analogous to the proof of Weyl's equidistribution theorem, to reduce matters to controlling the characteristic function $c_n(u,v)$.

Firstly, from \cite[Lemma 10.5]{bai} we have
$$ \frac{1}{n} \sum_{k=1}^n |\lambda_k|^2 \leq \frac{1}{n^2} \sum_{j,k=1}^n |\a_{jk}|^2.$$
Applying the Kolmogorov law of large numbers, we conclude that with probability $1$
$$ \frac{1}{n} \sum_{k=1}^n |\lambda_k|^2 \ll 1$$
for all sufficiently large $n$.  In particular, $d\mu_n$ assigns a measure of $O(n^{-\eta'})$ to the complement of the square $[-n^{\eta'/2}/2,n^{\eta'/2}/2]^2$ (one should think of this as a quantitative tightness estimate on $\mu_n$).  If we then let $d\tilde \mu_n$ be the push-forward of the measure $d\mu_n$ to the torus $(\R/n^{\eta'/2}\Z)^2$, and similarly define $d\tilde \mu_\infty$, it thus suffices to show that
\begin{eqnarray*}
\sup_{|s|,|t| \leq n^{\eta'/2}/2}
|&\int_{(\R/n^{\eta'/2}\Z)^2} 1_{[-n^{\eta'/2}/2,s] \times [-n^{\eta'/2}/2,t]}(s',t') d\tilde \mu_n(s',t') \\
&- \int_{(\R/n^{\eta'/2}\Z)^2}
 1_{[-n^{\eta'/2}/2,s] \times [-n^{\eta'/2}/2,t]}(s',t') d\tilde \mu_\infty(s',t')|\\ & \ll n^{-\eta'}.
 \end{eqnarray*}
Let $\varphi$ be a bump function adapted to the ball $B(0,n^{10\eta'})$, and let $\hat \varphi: (\R/n^{\eta'/2}\Z)^2 \to \C$ be the Fourier series
$$ \hat \varphi(s,t) := \frac{1}{n^{\eta'}} \sum_{u,v \in 2\pi n^{-\eta'/2} \Z} \varphi(u,v) e^{\sqrt{-1}us+\sqrt{-1}vt}.$$
This is an approximation to the identity, and one can then verify the pointwise bounds
$$1_{[-n^{\eta'/2}/2,s] \times [-n^{\eta'/2}/2,t]} \geq 1_{[-n^{\eta'/2}/2,s-n^{-5\eta'}] \times [-n^{\eta'/2}/2, t-n^{-5\eta'}]} * \hat \varphi + O(n^{-\eta'})$$
and
$$1_{[-n^{\eta'/2}/2,s] \times [-n^{\eta'/2}/2,t]} \leq 1_{[-n^{\eta'/2}/2,s+n^{-5\eta'}] \times [-n^{\eta'/2}/2, t+n^{-5\eta'}]} * \hat \varphi + O(n^{-\eta'})$$
(say).  Because of this, and the fact that $\tilde \mu_n$, $\tilde \mu_\infty$ are probability measures, it will suffice to show that
\begin{eqnarray*}
 \sup_{|s|,|t| \leq n^{\eta'/2}/2}
| & \int_{(\R/n^{\eta'/2}\Z)^2} 1_{[-n^{\eta'/2}/2,s] \times [-n^{\eta'/2}/2,t]}*\hat \varphi(s',t') d\tilde \mu_n(s',t')
\\ &- \int_{(\R/n^{\eta'/2}\Z)^2} 1_{[-n^{\eta'/2}/2,s] \times [-n^{\eta'/2}/2,t]}*\hat \varphi(s',t') d\tilde \mu_\infty(s',t')|  \\ &\ll n^{-\eta'}. \end{eqnarray*}
Taking Fourier transforms and using the triangle inequality, we can bound the left-hand side by
$$ n^{100\eta'} \sup_{u,v \in 2\pi n^{-\eta'/2} \Z: |u|, |v| \ll n^{10\eta'}}
|\int_{(\R/n^{\eta'/2}\Z)^2} e^{\sqrt{-1}us+\sqrt{-1}vt}d\tilde \mu_n(s,t) - \int_{(\R/n^{\eta'/2}\Z)^2} e^{\sqrt{-1}us+\sqrt{-1}vt}d\tilde \mu_\infty(s,t)|,$$
which is equal to
$$ n^{100\eta'} \sup_{u,v \in 2\pi n^{-\eta'/2} \Z: |u|, |v| \ll n^{10\eta'}} |c_n(u,v) - c(u,v)|.$$
Thus it will suffice (by the union bound and the Borel-Cantelli lemma) to show that for any fixed $u, v$ with $|u|, |v| \leq n^{\eta'}$, one has
\begin{equation}\label{cnc}
 \P( |c_n(u,v) - c(u,v)| \ll n^{-200\eta'} ) \geq 1 - O(n^{-10})
\end{equation}
for all sufficiently large $n$.

To prove \eqref{cnc} one repeats the proof of \eqref{cnu}, which
requires going through all the relevant arguments in \cite{bai} and
noting that all the almost sure convergence results can be replaced
instead with more quantitative polynomial convergence results
(similar to \eqref{cnc}).  We perform only one of these steps in
detail, namely the proof of the quantitative analogue of
\eqref{muso-2},
$$ \P( \int\int_T \left|\int_0^{\eps_n} \log x\ \nu_n(dx, z)\right| \ll n^{-200\eta'} ) \geq 1 - O(n^{-10}).$$
Inspecting the proof of \eqref{nux}, we see that for each fixed $z$, $\int_0^{\eps_n} \log x\ \nu_n(dx, z)$ vanishes with probability $O(n^{-100})$.  By Fubini's theorem and Markov's inequality, we thus see that with probability $1 - O(n^{-50})$, the set $\{ z \in T: \int_0^{\eps_n} \log x\ \nu_n(dx, z) \neq 0 \}$ has measure at most $n^{-50}$.  Since \eqref{intt} is bounded uniformly in $n$, the claim now follows from the Cauchy-Schwarz inequality.

\begin{remark} It is quite likely that one can make the convergence even more quantitative, establishing a bound of the form
$$ \P( \sup_{s,t} |\mu_n(s,t) - \mu_\infty(s,t)| \leq n^{-\eta'} ) \geq 1 - O( n^{-1-\eta'} )$$
for all $n \geq 1$; note that the claim \eqref{supmun} is a corollary of this bound and the Borel-Cantelli lemma.  This requires replacing the Kolmogorov law of large numbers with a more quantitative law of large numbers which takes advantage of the fact that the random variable $|\a|^2$ does not merely have finite first moment, but in fact has finite $(1 + \frac{\eta}{2})^\th$ moment.  We omit the details.
\end{remark}

\section{The sparse case}\label{sparse-sec}

In this section we sketch how one can modify the arguments in Section \ref{circular-sec} to obtain the circular law for sparse matrices (i.e. Theorem \ref{circular-sparse}).  The proof shall be a modification of that\footnote{It is also likely that the arguments in \cite{gotze} (see also \cite{PZ}) could also be adapted to handle this case, at least if one assumes additional moment conditions on $\a$, since the lower bound $\alpha>3/4$ required in that paper was only needed to obtain an analogue of Theorem \ref{lsv-sparse}.} of Theorem \ref{circular}.  In that theorem, one first needed the convergence
$$\lim_{n \to \infty} \frac{1}{n^2} \sum_{j,k=1}^n |\a_{jk}|^2 = \E |\a|^2 < \infty,$$
which was a consequence of the Kolmogorov law of large numbers, in order to obtain tightness of the $\mu_n$.  In the sparse case, the analogous convergence result one needs is
\begin{equation}\label{rhon}
\lim_{n \to \infty} \frac{1}{\rho n^2} \sum_{j,k=1}^n |\I_{j,k,\rho} \a_{jk}|^2 = \E |\a|^2 < \infty.
\end{equation}
But one easily computes that with probability $1$, $\I_{j,k,\rho}$ is equal to $1$ for $(1+o(1))\rho n^2$ values of $j,k$, and so this claim also follows from the Kolmogorov law of large numbers.

We now repeat the arguments of Section \ref{circular-sec}, using Theorem \ref{lsv-sparse} instead of Theorem \ref{lsv}.
The truncation argument in \cite[\S 10.5.1]{bai} which allows one to replace $\nu_n$ with $\tilde \nu_n$ can be easily modified, basically by similar arguments to the one used to deduce \eqref{rhon}.  The only step which requires care is the modification of \cite[Lemma 10.15]{bai} needed to establish the sparse analogue of \eqref{10.15}.  The proof of this lemma in \cite{bai} requires some upper bounds for the expected moments $\E \tr(\tilde H_n^k)$ of $\tilde H_n$ (see \cite[Lemma 10.11]{bai}), but it is not difficult to verify\footnote{As is well-known, the expected moments reduces to a sum over paths of length $k$, such as \eqref{eaa}.  For those paths in which each edge is traversed exactly twice, there is no difference between the sparse matrix and dense matrix as far as the expectation is concerned.  For those paths in which an edge is traversed more than twice, the sparse matrix contributes more than the dense matrix, but one can still show that the net contribution here is dominated by the main term in which each edge is traversed exactly twice; roughly speaking, for each fewer vertex that one traverses, one loses a factor of $n^\delta/\rho$ but picks up a factor of $n$, leading to a net gain of a positive power in $n$.}   that these upper bounds continue to hold in the sparse case.  The rest of the proof of \cite[Lemma 10.15]{bai} proceeds with only minor changes.

\subsection{Acknowledgements}

We are greatly indebted to Manjunath Krishnapur for pointing out
the connection between the least singular value bounds and  the
circular law problem and for many useful discussions, and to Dmitry
Timushev for corrections. We also thank P. Wood and an anonymous referee for their careful reading.
The first author also thanks Mark Rudelson
for some helpful conversations. The first author is supported by a
grant from the Macarthur Foundation and by NSF grant CCF-0649473.
The second author is supported by NSF Grant 06355606.

\end{document}